\documentclass[draft]{amsart}

\usepackage[pdfborder={0 0 0},draft=false,pagebackref]{hyperref}

\makeatletter\def\HyPsd@CatcodeWarning#1{}\makeatother

\usepackage{esint,amssymb,tikz,mathtools}

\newtheorem{thm}[equation]{Theorem}
\newtheorem{lem}[equation]{Lemma}

\theoremstyle{definition}
\newtheorem{dfn}[equation]{Definition}
\theoremstyle{remark}

\newtheorem{clm}[equation]{Claim}

\numberwithin{equation}{section}
\numberwithin{figure}{section}

\newcommand\abs[2][empty]{\csname#1\endcsname \lvert{#2}\csname#1\endcsname\rvert}
\newcommand\doublebar[2][empty]{\csname#1\endcsname \lVert{#2}\csname#1\endcsname\rVert}

\newcommand\mat[1]{\boldsymbol{#1}}
\newcommand\arr[1]{\boldsymbol{\dot{#1}}}

\newcommand\dist{\mathop{\mathrm{dist}}\nolimits}
\newcommand\Div{\mathop{\mathrm{div}}\nolimits}
\newcommand\Tr{\mathop{\smash{\arr{\mathrm{Tr}}}\vphantom{T}}\nolimits}

\newcommand\Trace{\mathop{\mathrm{Tr}}\nolimits}
\newcommand\M{\mathop{\smash{\arr{\mathrm{M}}}\vphantom{M}}\nolimits}

\newcommand\supp{\mathop{\mathrm{supp}}\nolimits}
\newcommand\diam{\mathop{\mathrm{diam}}\nolimits}

\newcommand\re{\mathop{\mathrm{Re}}\nolimits}

\newcommand\R{\mathbb{R}}

\newcommand\N{\mathbb{N}}
\newcommand\1{\mathbf{1}}

\newcommand\XX{\mathfrak{X}}
\newcommand\YY{\mathfrak{Y}}

\def\smooth{s}

\newcommand\pmin[1][\smooth]{\pdmnMinusOne/\allowbreak(\dmnMinusOne+#1)}

\def\dmn{{n+1}}
\def\pdmn{{(n+1)}}
\def\dmnMinusOne{n}

\def\pdmnMinusOne{n}

\def\dmn{d}
\def\pdmn{d}
\def\dmnMinusOne{{d-1}}

\def\pdmnMinusOne{{(d-1)}}

\usepackage{tikz}
\usetikzlibrary{calc,intersections,through}

\input Perturbation.aux

\begin{document}

\title[Extrapolation of well posedness]{Extrapolation of well posedness for higher order elliptic systems with rough coefficients}

\author{Ariel Barton}
\address{Ariel Barton, Department of Mathematical Sciences,
			309 SCEN,
			University of Ar\-kan\-sas,
			Fayetteville, AR 72701}
\email{aeb019@uark.edu}


\begin{abstract} 
In this paper we study boundary value problems for higher order elliptic differential operators in divergence form. We establish well posedness for problems with boundary data in Besov spaces $\dot B^{p,p}_\smooth$, $p\leq 1$, given well posedness for appropriate values of $\smooth$ and $p>1$. We work with smoothness parameter $\smooth$ between $0$ and~$1$; this allows us to consider inhomogeneous differential equations.

Combined with results of Maz'ya, I.~Mitrea, M.~Mitrea, and Shaposhnikova, this allows us to establish new well posedness results for higher order operators whose coefficients are in or close to the space $VMO$, for the biharmonic operator, and for fourth-order operators close to the biharmonic operator.

%
\end{abstract}


\maketitle

\setcounter{tocdepth}{1}
\tableofcontents
\setcounter{tocdepth}{3}


\section{Introduction}

In \cite{Bar16pA}, we studied higher order boundary value problems for elliptic differential operators $L$ of the form
\begin{equation}\label{eqn:divergence}
(L\vec u)_j =\sum_{k=1}^N \sum_{\abs{\alpha}=\abs{\beta}=m}\partial^\alpha (A^{jk}_{\alpha\beta}\,\partial^\beta u_k)\end{equation}
of arbitrary even order~$2m$,
for variable bounded measurable coefficients~$\mat A$. 
In particular, we studied the fully inhomogeneous Dirichlet problem
\begin{multline}
\label{eqn:Dirichlet:p:smooth}
L\vec u=\Div_m\arr H \text{ in }\Omega,\quad \Tr_{m-1}^\Omega\vec u=\arr f , 
\\
\doublebar{\vec u}_{\dot W^{p,\smooth}_{m,av}(\Omega)}\leq C \doublebar{\arr f}_{\dot W\!A^{p}_{m-1,\smooth}(\partial\Omega)} + C \doublebar{\arr H}_{L^{p,\smooth}_{av}(\Omega)}\end{multline}
and the Neumann problem
\begin{multline}
\label{eqn:Neumann:p:smooth}
L\vec u=\Div_m\arr H \text{ in }\Omega,\quad \M_{\mat A,\arr H}^\Omega \vec u=\arr g, 
\\
\doublebar{\vec u}_{\dot W^{p,\smooth}_{m,av}(\Omega)}\leq C \doublebar{\arr g}_{\dot N\!A^{p}_{m-1,\smooth-1}(\partial\Omega)} + C \doublebar{\arr H}_{L^{p,\smooth}_{av}(\Omega)}
\end{multline}
for $p$, $\smooth$ with $0<\smooth<1$ and $\pmin<p\leq\infty$.

In this paper we will extrapolate well posedness from the range $p>1$ to the range $p\leq 1$; this will yield new well posedness results in the case where $\mat A$ lies in the space $VMO$ 
and in the case of the biharmonic operator $L=\Delta^2$. See Section~\ref{sec:A0:introduction}. 

The Dirichlet boundary values $\Tr_{m-1}^\Omega\vec u$ of $\vec u$ in the problem~\eqref{eqn:Dirichlet:p:smooth} are given by $\Tr_{m-1}^\Omega \vec u=\Trace^\Omega \nabla^{m-1}\vec u$, where $\Trace$ is the standard trace operator of Sobolev spaces; see \cite[Definition~2.4]{Bar16pB}. The Neumann boundary values $\M_{\mat A,\arr H}^\Omega \vec u$ are given by 
\begin{equation}\label{dfn:Neumann}
\langle\nabla^m\vec \varphi,\mat A\nabla^m \vec u-\arr H\rangle_\Omega=\langle \Tr^\Omega_{m-1}\vec \varphi, \M_{\mat A,\arr H}^\Omega \vec u\rangle_{\partial\Omega} 
\text{ for all $\vec\varphi\in C^\infty_0(\R^\dmn)$}
\end{equation}
where $\langle\,\cdot\,,\,\cdot\,\rangle_\Omega$ denotes the standard $L^2$ inner product in~$\Omega$. 
The standard weak formulation of the operator $L$ given by formula~\eqref{eqn:divergence} is 
\begin{equation}\label{eqn:weak}
L\vec u=\Div_m\arr H\text{ in $\Omega$ if }
\langle \nabla^m\vec\varphi, \mat A\nabla^m\vec u \rangle_\Omega
=\langle \nabla^m\vec\varphi, \arr H\rangle_\Omega
\text{ for all $\vec\varphi\in C^\infty_0(\Omega)$}.\end{equation}
It may readily be seen that if $\partial\Omega$ is connected and if $L\vec u=\Div_m\arr H$ in $\Omega$ in the weak sense,
then the left-hand side of formula~\eqref{dfn:Neumann} depends only on $\Tr_{m-1}^\Omega\vec\varphi$ and not on the interior values of~$\vec\varphi$, and so $\smash{\M_{\mat A,\arr H}^\Omega \vec u}$ is a well defined operator on $\{\Tr_{m-1}^\Omega\vec \varphi:\vec\varphi\in C^\infty_0(\R^\dmn)\}$. This was the notion of Neumann boundary values used in \cite{Bar16pA,BarHM15p,BarHM17pC}, and is similar to but subtly different from that of \cite{CohG85,Ver05,Agr07,MitM13A}; see \cite{Ver10,BarM16B,BarHM15p} for a more extensive discussion of various notions of Neumann boundary values.

The boundary spaces ${\dot W\!A^{p}_{m-1,\smooth}(\partial\Omega)}$ and ${\dot N\!A^{p}_{m-1,\smooth-1}(\partial\Omega)}$ are subspaces and quotient spaces, respectively, of the vector-valued Besov spaces $\smash{(\dot B^{p,p}_\smooth(\partial\Omega))^{Nq}}$ and $(\dot B^{p,p}_{\smooth-1}(\partial\Omega))^{Nq}$, where $q$ denotes the number of multiindices in~$\N_0^\dmn$ of length~$m-1$. Specifically, ${\dot W\!A^{p}_{m-1,\smooth}(\partial\Omega)}$ is the closure in $\smash{(\dot B^{p,p}_\smooth(\partial\Omega))^{Nq}}$ of $\{\smash{\Tr_{m-1}^\Omega\vec\varphi}:\vec\varphi\in C^\infty_0(\R^\dmn)\}$, while ${\dot N\!A^{p}_{m-1,\smooth-1}(\partial\Omega)}=(\dot B^{p,p}_{\smooth-1}(\partial\Omega))^{Nq}/\sim$ is a quotient space under the equivalence relation $\arr g\sim\arr\gamma$ if $\langle \Tr_{m-1}^\Omega\vec\varphi,\arr g\rangle_{\partial\Omega}=\langle \Tr_{m-1}^\Omega\vec\varphi,\arr \gamma\rangle_{\partial\Omega}$ for all $\vec\varphi\in C^\infty_0(\R^\dmn)$. See \cite{Bar16pB} for a more detailed definition of these function spaces.

The space $\dot W^{p,\smooth}_{m,av}(\Omega)$ is the set of (equivalence classes of) functions $\vec u$ for which the $\dot W^{p,\smooth}_{m,av}(\Omega)$-norm given by
\begin{align}\label{eqn:W:norm:2}
\doublebar{\vec u}_{\dot W^{p,\smooth}_{m,av}(\Omega)}&=\doublebar{\nabla^m \vec u}_{ L^{p,\smooth}_{av}(\Omega)},
\\
\label{eqn:L:norm:2}
\doublebar{\arr H}_{L^{p,\smooth}_{av}(\Omega)}
&= 
\biggl(\int_\Omega \biggl(\fint_{B(x,\dist(x,\allowbreak \partial\Omega)/2)} \abs{\arr H}^2 \biggr)^{p/2}  \dist(x,\allowbreak \partial\Omega)^{p-1-p\smooth}\,dx\biggl)^{1/p}
\end{align}
is finite. (Two functions are equivalent if their difference has norm zero; if $\Omega$ is open and connected then two functions are equivalent if they differ by a polynomial of degree at most $m-1$.)
Here $\fint$ denotes the averaged integral $\fint_B H = \frac{1}{\abs{B}}\int_B H$. 

We refer the reader to \cite[Section~1.1]{Bar16pA} for a more extensive discussion of the historical significance of these function spaces. Here we will merely mention that the spaces ${\dot W\!A^{p}_{m-1,\smooth}(\partial\Omega)}$ and ${\dot N\!A^{p}_{m-1,\smooth-1}(\partial\Omega)}$ are commonly studied spaces of boundary data with fractional orders of smoothness (that is, smoothness parameters between zero and one). The theory of boundary value problems with data in integer smoothness spaces (that is, in Lebesgue or Sobolev spaces) is extensive; however, this theory is tightly focused on homogeneous differential equations $L\vec u=0$ rather than the inhomogeneous equations $L\vec u=\Div_m \arr H$ considered here, and to study inhomogeneous problems we generally must consider boundary data in fractional smoothness spaces. The spaces $\dot W^{p,\smooth}_{m,av}(\Omega)$ and $L^{p,\smooth}_{av}(\Omega)$ are connected by trace and extension theorems to the spaces ${\dot W\!A^{p}_{m-1,\smooth}(\partial\Omega)}$ and ${\dot N\!A^{p}_{m-1,\smooth-1}(\partial\Omega)}$; see \cite{Bar16pB}. Moreover, these spaces are well adapted to the theory of operators with rough coefficients; see \cite[Remark~10.9]{BarM16A}. 
We refer the reader also to \cite{JerK95,AdoP98} for some early appearances of weighted Sobolev norms $\int_\Omega \abs{\nabla^m\vec u(x)}^p\dist(x,\allowbreak \partial\Omega)^{p-1-p\smooth}\,dx$, \cite{MazMS10} for well posedness results given explicitly in terms of weighted Sobolev norms, and \cite[Section~\ref*{P:sec:introduction:history}]{Bar16pA}  for a discussion of the significance of averaged norms.

The main result of the present paper is the following theorem extrapolating well posedness from $p>1$ to $p\leq 1$. In Section~\ref{sec:A0:introduction}, we will combine this result with known results from the literature to derive new well posedness results.

\begin{thm}\label{thm:atoms} 
Let $\Omega\subset\R^\dmn$ be a Lipschitz domain with connected boundary.
Let $L$ be an operator of the form \eqref{eqn:divergence}, defined in the weak sense of formula~\eqref{eqn:weak}, associated to coefficients~$\mat A$ that satisfy the bound
\begin{align}
\label{eqn:elliptic:bounded}
\doublebar{\mat A}_{L^\infty(\R^\dmn)}
&\leq 
	\Lambda
\end{align}
and the G\r{a}rding inequality
\begin{align}
\label{eqn:elliptic:everywhere}
\re {\bigl\langle\nabla^m\vec \varphi,\mat A \nabla^m\vec \varphi\bigr\rangle_{\R^\dmn}} 
&\geq 
	\lambda\doublebar{\nabla^m\vec\varphi}_{L^2(\R^\dmn)}^2
	\quad\text{for all $\vec\varphi\in\dot W^2_m(\R^\dmn)$}
.\end{align}


Suppose that there exists a $\sigma_-$ and $q_-$ with 
\begin{equation}\label{eqn:q:sigma:1}0<\sigma_-<1,\quad 1< q_- < \frac{\dmnMinusOne}{\dmn-2+\sigma_-}, \quad q_-\leq 2\end{equation}
such that the Dirichlet problem
\begin{equation}\label{eqn:Dirichlet:q-}L\vec u=\Div_m\arr \Phi\text{ in }\Omega,\quad \Tr_{m-1}^\Omega\vec u=0,\quad \doublebar{\vec u}_{\dot W^{q_-,\sigma_-}_{m,av}(\Omega)} \leq C_0\doublebar{\arr \Phi}_{L^{q_-,\sigma_-}_{av}(\Omega)}\end{equation}
is well posed.

Let $M$ be the primary Lipschitz constant of~$\Omega$.
Suppose that there are some constants $M_0$, $\widetilde \sigma_+$, and $\widetilde q_+$ such that
\begin{equation}\label{eqn:q:sigma:2}
M_0>M,\quad \sigma_-+\pdmnMinusOne\biggl(1-\frac{1}{q_-}\biggr)<\widetilde\sigma_+<1, \quad 1< \widetilde q_+\leq 2\end{equation}
and such that, if $ T $ is a 
bounded simply connected Lipschitz domain  with primary Lipschitz constant at most~$M_0$, and if $\vec u\in {\dot W^{\widetilde q_+,\widetilde \sigma_+}_{m,av}( T )}$ with $L\vec u=0$ in $ T $, then
\begin{equation}\label{eqn:Dirichlet:tilde}
\doublebar{\vec u}_{\dot W^{\widetilde q_+,\widetilde \sigma_+}_{m,av}( T )} \leq C_1\doublebar{\Tr_{m-1}^{ T }\vec u}_{\dot W\!A^{\widetilde q_+}_{m-1,\widetilde \sigma_+}(\partial T )}
\end{equation}
where $C_1$ depends only on $\widetilde q_+$, $\widetilde \sigma_+$, $L$, and the (full) Lipschitz character of~$ T $.

Finally, suppose that there are positive numbers $\sigma_+$ and $q_+$ that satisfy
\begin{multline}\label{eqn:q:sigma:3}\widetilde\sigma_+\leq \sigma_+<1,\quad
1<q_+<\infty,\quad
 \widetilde \sigma_+-\frac{\dmnMinusOne}{\widetilde q_+}
\leq
\sigma_+-\frac{\dmnMinusOne}{q_+},\\\text{and if $\widetilde \sigma_+=\sigma_+$ then $q_+=\widetilde q_+$,}\end{multline}
such that the Dirichlet problem
\begin{equation}\label{eqn:Dirichlet:q+}L\vec u=\Div_m\arr \Phi\text{ in }\Omega,\quad \Tr_{m-1}^\Omega\vec u=0,\quad \doublebar{\vec u}_{\dot W^{q_+,\sigma_+}_{m,av}(\Omega)} \leq C_0\doublebar{\arr \Phi}_{L^{q_+,\sigma_+}_{av}(\Omega)}\end{equation}
is well posed. If $\Omega$ is unbounded we impose the additional assumption that the problems \eqref{eqn:Dirichlet:q-} and~\eqref{eqn:Dirichlet:q+} are compatibly well posed in the sense of \cite[Lemma~\ref*{P:lem:interpolation}]{Bar16pA}.

Then there exist numbers $p$ and $s$ that satisfy the condition
\begin{equation}
\label{eqn:q:smooth}
\sigma_-<s<\widetilde\sigma_+,\quad 0<p\leq 1,\quad \sigma_--\frac{\dmnMinusOne}{q_-}
<s-\frac{\dmnMinusOne}{p} 
\end{equation}
and for every such $s$ and~$p$ the Dirichlet problem~\eqref{eqn:Dirichlet:p:smooth}
is well posed, where $C$ depends only on $\lambda$, $\Lambda$, the Lipschitz character of~$\Omega$, $p$, $\smooth$, $q_\pm$, $\sigma_\pm$, $\widetilde q_+$, $\widetilde \sigma_+$, $C_0$, $M_0$, and $C_1$.


Similarly, suppose that $\mat A$ satisfies the bound~\eqref{eqn:elliptic:bounded} and the local G\r{a}rding inequality
\begin{align}
\label{eqn:elliptic:domain}
\re {\bigl\langle\nabla^m\vec \varphi,\mat A \nabla^m\vec \varphi\bigr\rangle_\Omega} 
&\geq 
	\lambda\doublebar{\nabla^m\vec\varphi}_{L^2(\Omega)}^2
	\quad\text{for all $\vec\varphi\in\dot W^2_m(\Omega)$}
.\end{align}
Let $M_0$, $q_\pm$, $\sigma_\pm$, $\widetilde q_+$, and $\widetilde \sigma_+$ satisfy the conditions \eqref{eqn:q:sigma:1}, \eqref{eqn:q:sigma:2} and~\eqref{eqn:q:sigma:3}.
Suppose that the two Neumann problems
\begin{equation}\label{eqn:Neumann:q-}L\vec u=\Div_m\arr \Phi\text{ in }\Omega,\quad \M_{\mat A,\arr \Phi}^\Omega\vec u=0,\quad \doublebar{\vec u}_{\dot W^{q_\pm,\sigma_\pm}_{m,av}(\Omega)} \leq C_0\doublebar{\arr \Phi}_{L^{q_\pm,\sigma_\pm}_{av}(\Omega)}\end{equation}
are compatibly well posed.
Suppose that there is some $M_0>M$ such that, if $ T $ is a  bounded simply connected Lipschitz domain with primary Lipschitz constant at most~$M_0$, and if $\vec u\in {\dot W^{\widetilde q_+,\widetilde \sigma_+}_{m,av}( T )}$ and $
L\vec u=0$ in $ T $, then
\begin{equation}\label{eqn:Neumann:tilde}
\doublebar{\vec u}_{\dot W^{\widetilde q_+,\widetilde \sigma_+}_{m,av}( T )} \leq  C_1\doublebar{\M_{\mat A,0}\vec u}_{\dot N\!A^{\widetilde q_+}_{m-1,\widetilde\sigma_+-1}(\partial T )}
\end{equation}
where $C_1$ depends only on $\widetilde q_+$, $\widetilde\sigma_+$, $L$, and the Lipschitz character of~$ T $.
Then for any $p$ and $\smooth$ that satisfy the bounds~\eqref{eqn:q:smooth}, the Neumann problem~\eqref{eqn:Neumann:p:smooth}
is well posed, where $C$ has the same dependencies as before.
\end{thm}
We will define the primary Lipschitz constant of~$\Omega$ in Definition~\ref{dfn:Lipschitz}.



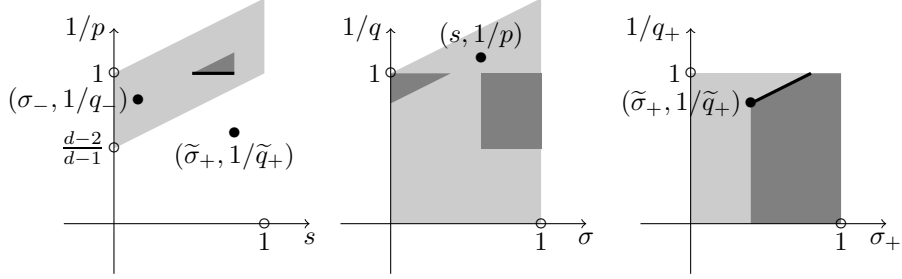
\begin{figure}
\begin{center}

\def\enn{2}

\begin{tikzpicture}[scale=2]
\def\Sm{0.16}
\def\Qm{0.82}
\def\Qt{0.6}
\def\St{0.8}

\fill [white!80!black] (0,1-1/\enn) -- (1,1) -- (1,1+1/\enn) -- (0,1) -- cycle;

\fill [white!50!black] (\Sm+\enn-\enn*\Qm,1) -- (\St,1)--(\St,\Qm+\St/\enn-\Sm/\enn) --cycle;
\draw [very thick] (\Sm+\enn-\enn*\Qm,1) -- (\St,1);

\draw [->] (-1/3,0)--(1.3,0) node [below] {$\smooth$};
\draw [->] (0,-1/3)--(0,1.3) node [left] {$1/p$};
\path (0,1) node {$\circ$} node [left] {$1$};
\path (1,0) node {$\circ$} node [below] {$1$};


\path (0,1-1/\enn) node {$\circ$} node [left] {$\frac{\dmn-2}{\dmnMinusOne}$};

\path (\St,\Qt) node {$\bullet$} node [below] {$(\widetilde\sigma_+,1/\widetilde q_+)$};
\path (\Sm,\Qm) node {$\bullet$} node [left] {$(\sigma_- ,1/ q_-)$};

\end{tikzpicture}
\begin{tikzpicture}[scale=2]

\def\P{1.1}
\def\Th{0.6}

\fill [white!80!black] (0,0) -- (1,0) -- (1,1+1/\enn) -- (0,1) -- cycle;

\fill [white!50!black] (0,\P-\Th/\enn) -- (0,1) -- (\Th-\enn*\P+\enn,1) -- cycle;
\fill [white!50!black] (\Th,1) -- (1,1)--(1,1/2)--(\Th,1/2) -- cycle;

\node at (\Th,\P) {$\bullet$};
\node at (\Th,\P) [above] {$(\smooth,1/p)$};

\draw [->] (-1/3,0)--(1.3,0) node [below] {$\sigma$};
\draw [->] (0,-1/3)--(0,1.3) node [left] {$1/q$};
\node at (0,1) {$\circ$}; \node at (0,1) [left] {$1$};
\node at (1,0) {$\circ$}; \node at (1,0) [below] {$1$};

\end{tikzpicture}
\def\Qt{0.8}%
\def\St{0.4}%
\begin{tikzpicture}[scale=2]

\fill [white!80!black] (0,0) -- (1,0) -- (1,1) -- (0,1) -- cycle;

\fill [white!50!black] (\St,\Qt) -- (\St+\enn-\enn*\Qt,1) --(1,1)-- (1,0)--(\St,0) --cycle;
\draw [very thick] (\St,\Qt) -- (\St+\enn-\enn*\Qt,1);

\draw [->] (-1/3,0)--(1.3,0) node [below] {$\sigma_+$};
\draw [->] (0,-1/3)--(0,1.3) node [left] {$1/ q_+$};
\path (0,1) node {$\circ$} node [left] {$1$};
\path (1,0) node {$\circ$} node [below] {$1$};

\path (\St,\Qt) node {$\bullet$} node [left] {$(\widetilde\sigma_+,1/\widetilde q_+)$};

\end{tikzpicture}
\end{center}
\caption{Theorem~\ref{thm:atoms}: On the left, we show the acceptable values of $(s,1/p)$, given $(\sigma_-,1/q_-)$ and $(\widetilde\sigma_+,1/\widetilde q_+)$. In the middle, we show the acceptable values of $(\sigma_-,1/q_-)$ and $(\widetilde\sigma_+,1/\widetilde q_+)$, given $(s,1/p)$. Finally, on the right, we show the acceptable values of $(\sigma_+,1/q_+)$, given $(\widetilde\sigma_+,1/\widetilde q_+)$.
}
\label{fig:atoms}
\end{figure}

The problems \eqref{eqn:Dirichlet:q-} and~\eqref{eqn:Dirichlet:q+} (or the two problems~\eqref{eqn:Neumann:q-}) are compatibly well posed if, for every $\arr H\in L^{q_-,\sigma_-}_{av}(\Omega)\cap L^{q_+,\sigma_+}_{av}(\Omega)$, there is a single function $\vec u\in \dot W^{q_-,\sigma_-}_{m,av}(\Omega)\cap \dot W^{q_+,\sigma_+}_{m,av}(\Omega)$ that is a solution to both of the problems \eqref{eqn:Dirichlet:q-} and~\eqref{eqn:Dirichlet:q+} (or both of the problems~\eqref{eqn:Neumann:q-}). Compatibility of solutions is not trivial; the main result of \cite{Axe10} is an example of a second order operator $L$ such that the Dirichlet problems
\begin{align*}Lu=0 \text{ in }\R^2_+,\quad \Trace u &= f, \quad \doublebar{u}_{T^p_\infty}\leq \doublebar{f}_{L^p(\partial\R^2_+)}
,\\
Lv=0 \text{ in }\R^2_+,\quad \Trace v &= f, \quad \doublebar{v}_{\dot W^2_1(\R^2_+)}\leq \doublebar{f}_{\dot B^{2,2}_{1/2}(\partial\R^2_+)}\end{align*}
are both well posed, but for which $u\neq v$ for some $f\in {L^p(\partial\R^2_+)}\cap{\dot B^{2,2}_{1/2}(\partial\R^2_+)}$. Here $T^p_\infty$ is the tent space defined in \cite{CoiMS85}.

The proof of Theorem~\ref{thm:atoms} (see in particular Section~\ref{sec:atoms:boundary}) uses many techniques from the proofs of \cite[Lemma~1.6]{DahK90} and~\cite[Theorem~9.6]{PipV92}. 
Both of these papers treat specific operators~$L$ (the Lam\'e system in \cite{DahK90}, the biharmonic operator in \cite{PipV92}) and establish well posedness of the homogeneous Dirichlet problem with boundary data in the Hardy space $\dot H^1_1(\partial\Omega)$; this space is closely related to the $p=1$, $\smooth=1$ endpoint of the spaces $\dot W\!A^p_{m-1,\smooth}(\partial\Omega)$ used in this paper. (\cite[Lemma~1.6]{DahK90} also considered the traction boundary problem, that is, the Neumann problem.)
We have adapted their proofs to the case of inhomogeneous equations $L\vec u=\Div_m\arr H$, and thus to the case of boundary data in fractional smoothness spaces.

We now turn to the history of the boundary value problems 
\begin{alignat}{6}
\label{eqn:Dirichlet:general}
L\vec u&=\Div_m\arr H\text{ in }\Omega,&\>\>\> \Tr_{m-1}\vec u&=\arr f, &\>\>\> \doublebar{\vec u}_\XX&\leq C\doublebar{\arr f}_{\dot W\!A^{p}_{m-1,\smooth}(\partial\Omega)}+C\doublebar{\arr H}_\YY
\\
L\vec u&=\Div_m\arr H\text{ in }\Omega,&\>\>\> \M_{\mat A,\arr H}^\Omega\vec u&=\arr g, &\>\>\> \doublebar{\vec u}_\XX&\leq C\doublebar{\arr g}_{\dot N\!A^{p}_{m-1,\smooth-1}(\partial\Omega)}+C\doublebar{\arr H}_\YY
\end{alignat}
for some appropriate spaces $\XX$ and~$\YY$, with $0<\smooth<1$ and with $p\leq 1$. 

In \cite{MayMit04A}, well posedness was established in the case $L=\Delta$ (that is, for Laplace's and Poisson's equations) in Lipschitz domains for $1-\kappa<\smooth<1$ and $\pdmnMinusOne/(\dmn-2+\kappa+\smooth)<p\leq 1$. In the case $\dmn=2$ they showed that $\kappa>1/2$, but if $\dmn\geq 3$ then they showed only $\kappa>0$. In dimensions $\dmn=2$ and $\dmn=3$, this is the range of $p$, $\smooth$ such that well posedness of the problems~\eqref{eqn:Dirichlet:p:smooth} and~\eqref{eqn:Neumann:p:smooth} may be established via Theorem~\ref{thm:atoms} and earlier results  in \cite{JerK95,FabMM98,Zan00}. If $\dmn\geq 4$, then the results of \cite{MayMit04A} are stronger; in fact in this case Theorem~\ref{thm:atoms} does not yield any well posedness results.

In \cite{MitMW11}, well posedness of the Dirichlet problem \eqref{eqn:Dirichlet:general} for the biharmonic operator $L=\Delta^2$ was established in three-dimensional Lipschitz domains, again for $1-\kappa<s<1$ and $\pdmnMinusOne/(\dmn-2+\kappa+s)<p\leq 1$. In Section~\ref{sec:biharmonic:introduction}, we will establish well posedness results for the two-dimensional biharmonic problems and the three-dimensional Neumann problem, working from Theorem~\ref{thm:atoms} and known results of \cite{MitM13B}; the results for the three-dimensional Dirichlet problem derived in the same way are essentially equivalent to those of \cite{MitMW11}.

If $\Omega$ is a two dimensional $VMO$ domain, then well posedness of the biharmonic Dirichlet problem for $0<\smooth<1$ and $1/(1+\smooth)<p\leq 1$ was established in \cite[Theorem~6.35]{MitM13A}. (The case of  $C^1$ domains was also considered in \cite{MayMit04A}; in that case the parameter $\kappa$ may be taken to be $1$ and so the range of well posedness is again $\pmin<p\leq 1$.)

We turn to the case of variable coefficients. In \cite{BarM16A}, Mayboroda and the author of the present paper investigated well posedness in the case
where $N=m=1$, where $\mat A$ is real and $t$-independent in the sense that $\mat A(x',t)=\mat A(x',s)$ for all $s$, $t\in\R$ and all $x'\in\R^\dmnMinusOne$, and where $\Omega$ is upper half space, or more generally  a Lipschitz graph domain $\Omega=\{(x',t):x'\in\R^\dmnMinusOne,\allowbreak\>t>\psi(x')\}$ for a Lipschitz function~$\psi$.
In this case, the Dirichlet problem \eqref{eqn:Dirichlet:p:smooth} is well posed whenever $1-\kappa<\smooth<1$ and $\pdmnMinusOne/(\dmn-2+\kappa+\smooth)<p\leq 1$, for some $\kappa>0$. If in addition $\mat A$ is symmetric then the Neumann problem is well posed for the same range of $\smooth$ and~$p$. As observed in \cite[Remark~\ref*{P:rmk:BarM16A:2}]{Bar16pA}, if $\mat A$ is symmetric and $\dmn=2$ then $\kappa>1/2$. 

Finally, in \cite{Bar16pA}, we established stability of well posedness under $L^\infty$ perturbation of the coefficients. As a consequence, if $\dmn=3$ and $\mat A$ is close (in $L^\infty$) to the coefficients of the biharmonic operator, or if $m=1$ and $\mat A$ is close to being $t$-independent, then the Dirichlet problem \eqref{eqn:Dirichlet:p:smooth} is well posed for $1-\kappa+\varepsilon<\smooth<1-\varepsilon$ and $\pdmnMinusOne/(\dmn-2+\kappa-\varepsilon+\smooth)<p\leq 1$, for some $0<\varepsilon<\kappa/2$. (Corresponding results are valid for the Neumann problem if $\mat A$ is close to being symmetric.)


\subsection{New well posedness results derived from Theorem~\ref{thm:atoms}}
\label{sec:A0:introduction}

In this section we review some known well posedness results from the literature, and we discuss how these well posedness results combine with Theorem~\ref{thm:atoms} to yield new well posedness results.

\subsubsection{Systems with $VMO$ coefficients}
\label{sec:MazMS10:introduction}

In \cite{MazMS10}, Maz'ya, Mitrea and Shaposhnikova investigated the Dirichlet problem for coefficients in (or close to) the space $VMO$. In the terminology of the present paper, their results may be shown to be equivalent to the following. See Section~\ref{sec:A0}.

Suppose that $\mat A$ is bounded and elliptic in the sense of satisfying the conditions \eqref{eqn:elliptic:bounded} and~\eqref{eqn:elliptic:everywhere}, and that $\mat A\in VMO(\R^\dmn)$.
Suppose furthermore that $\Omega\subset\R^\dmn$ is a bounded Lipschitz domain with connected boundary, and that the unit outward normal~$\nu$ to $\partial\Omega$ lies in~$VMO(\partial\Omega)$.
Then the Dirichlet problem \eqref{eqn:Dirichlet:p:smooth}
is well posed for any $0<\smooth<1$ and any $1<p<\infty$.

More generally, let
\begin{multline}
\label{eqn:MazMS10}
\delta(\mat A,\Omega) 
= \lim_{\delta\to 0^+} \sup_{x\in\Omega} \sup_{0<r<\delta} \fint_{B(x,r)\cap\Omega}\fint_{B(x,r)\cap\Omega} \abs{\mat A(y)-\mat A(z)}\,dz\,dy
\\+\lim_{\delta\to 0^+} \sup_{x\in\partial\Omega} \sup_{0<r<\delta} \fint_{B(x,r)\cap\partial\Omega}\fint_{B(x,r)\cap\partial\Omega} \abs{\nu(y)-\nu(z)}\,d\sigma(z)\,d\sigma(y).
\end{multline}
There is some constant $c$ such that, if $1<p<\infty$ and $0<\smooth<1$, and if
\begin{equation*}\delta(\mat A,\Omega) <
c\min(\smooth,1-\smooth,1/p,1-1/p)^2
\end{equation*}
then the Dirichlet problem~\eqref{eqn:Dirichlet:p:smooth} is well posed. 

We remark that by Definition~\ref{dfn:Lipschitz}, if $r$ is small enough then
\begin{equation*}\fint_{B(x,r)\cap\partial\Omega}\fint_{B(x,r)\cap\partial\Omega} \abs{\nu(y)-\nu(z)}\,d\sigma(z)\,d\sigma(y)\leq 2M\end{equation*}
where $M$ is the primary Lipschitz constant of~$\Omega$ mentioned in Theorem~\ref{thm:atoms}.

Combined with Theorem~\ref{thm:atoms} (for $p\leq 1$), \cite[Lemma~\ref*{P:lem:interpolation}]{Bar16pA} (for $1<p<1/(1-\varepsilon)$, and \cite[Lemma~\ref*{P:lem:BVP:duality}]{Bar16pA} (for $1/\varepsilon<p\leq\infty$), we have the following well posedness result.
\begin{thm}\label{thm:MazMS10:p<1}
Fix some $\varepsilon$ with $0<\varepsilon<1/2$. Then there is some $\delta_0>0$ such that, if $\mat A$ is bounded and satisfies the ellipticity condition~\eqref{eqn:elliptic:everywhere}, if $\Omega$ is a bounded simply connected Lipschitz domain with primary Lipschitz constant~$M$, and if 
\begin{equation*}2M+\lim_{\delta\to 0^+} \sup_{x\in\R^\dmn} \sup_{0<r<\delta} \fint_{B(x,r)}\fint_{B(x,r)} \abs{\mat A(y)-\mat A(z)}\,dz\,dy<\delta_0,\end{equation*}
then the Dirichlet problem \eqref{eqn:Dirichlet:p:smooth} is well posed whenever 
\begin{equation*}\varepsilon<\smooth<1-\varepsilon, \quad 0<p\leq \infty,\quad
\frac{\smooth-1+\pdmn\varepsilon}{\dmnMinusOne}<\frac{1}{p} < 1+\frac{\smooth-\pdmn\varepsilon}{\pdmnMinusOne}.\end{equation*}

In particular, if $\mat A\in VMO(\R^\dmn)$ and $\nu\in VMO(\partial\Omega)$, then the Dirichlet problem \eqref{eqn:Dirichlet:p:smooth} is well posed whenever $0<\smooth<1$ and $\pmin<p\leq \infty$.
\end{thm}
The acceptable values of $\smooth$ and $p$ are shown in Figure~\ref{fig:MazMS10}.

\begin{figure}
\begin{center}
\begin{tikzpicture}[scale=2]

\def\eps{0.1}
\def\enn{2}

\fill [white!80!black] (0,0) -- (1,0) -- (1,1) -- (0,1) -- cycle;

\draw [->] (-1/3,0)--(1.3,0) node [below] {$\smooth$};
\draw [->] (0,-1/3)--(0,1.3) node [left] {$1/p$};

\node at (0,1) {$\circ$}; \node at (0,1) [left] {$1$};
\node at (1,0) {$\circ$}; \node at (1,0) [below] {$1$};

\fill [white!50!black] (\eps,\eps) -- (\eps,1-\eps) -- (1-\eps,1-\eps) -- (1-\eps,\eps) -- cycle ;

\end{tikzpicture}
\begin{tikzpicture}[scale=2]

\def\eps{0.1}
\def\enn{2}

\fill [white!80!black] (0,0) -- (1,0) -- (1,1+1/\enn) -- (0,1) -- cycle;

\draw [->] (-1/3,0)--(1.3,0) node [below] {$\smooth$};
\draw [->] (0,-1/3)--(0,1.3) node [left] {$1/p$};

\node at (0,1) {$\circ$}; \node at (0,1) [left] {$1$};
\node at (1,0) {$\circ$}; \node at (1,0) [below] {$1$};
\node at (1,1+1/\enn) {$\circ$}; \node at (1,1+1/\enn) [right] {$(1,\pdmn/\pdmnMinusOne)$};

\fill [white!50!black] (\enn*\eps+\eps,1) -- (1-\eps,1) -- (1-\eps, 1-\eps+1/\enn-2*\eps/\enn) -- cycle;
\draw [very thick] (\enn*\eps+\eps,1) -- (1-\eps,1);

\fill [white!65!black] (\eps,\eps) -- (\eps,1-\eps) -- (1-\eps,1-\eps) -- (1-\eps,\eps) -- cycle ;

\end{tikzpicture}
\begin{tikzpicture}[scale=2]

\def\eps{0.1}
\def\enn{2}

\fill [white!80!black] (0,0) -- (1,0) -- (1,1+1/\enn) -- (0,1) -- cycle;

\draw [->] (-1/3,0)--(1.3,0) node [below] {$\smooth$};
\draw [->] (0,-1/3)--(0,1.3) node [left] {$1/p$};

\node at (0,1) {$\circ$}; \node at (0,1) [left] {$1$};
\node at (1,0) {$\circ$}; \node at (1,0) [below] {$1$};

\fill [white!50!black] (\eps,1-\eps) -- (\eps,0) -- (1-\eps-\eps*\enn,0) -- (1-\eps,\eps) -- (1-\eps, 1-\eps+1/\enn-2*\eps/\enn) -- cycle;
\draw [very thick] (\eps,0) -- (1-\eps-\eps*\enn,0);

\end{tikzpicture}
\end{center}
\caption{Theorem~\ref{thm:MazMS10:p<1}: If $\mat A$ is sufficiently close to $VMO$ and $\Omega$ is sufficiently close to a $VMO$ domain, 
well posedness was established in \cite{MazMS10} whenever $(\smooth,1/p)$ lies in the square on the left. Theorem~\ref{thm:atoms} allows us to extrapolate to well posedness whenever $(\smooth,1/p)$ lies in the triangle in the middle, and \cite[Lemmas \ref*{P:lem:BVP:duality} and~\ref*{P:lem:interpolation}]{Bar16pA} allow us to establish well posedness whenever $(\smooth,1/p)$ lies in the pentagon on the right.}\label{fig:MazMS10}
\end{figure}

Recall that in the special case where $L=\Delta$ and $\Omega$ is~$C^1$, or where $L=\Delta^2$ is the biharmonic operator and the ambient dimension $\dmn=2$, this result is very similar to a known result; see \cite[Section~3]{MayMit04A} and
\cite[Theorem~6.35]{MitM13A}.

\subsubsection{The biharmonic equation and perturbations}
\label{sec:biharmonic:introduction}

In \cite{MitM13B}, I.~Mitrea and M. Mit\-rea established well posedness of boundary value problems for the biharmonic operator $\Delta^2$ in bounded Lipschitz domains. Their results improved upon previous results of Adolfsson and Pipher \cite{AdoP98}. They worked with solutions $u$ in Besov and Triebel-Lizorkin spaces. We prefer to work with solutions $u$ in the weighted averaged spaces $\dot W^{p,\smooth}_{2,av}(\Omega)$; thus, we will use the following result (derived from the results of \cite{MitM13B} in \cite[Section~\ref*{P:sec:biharmonic}]{Bar16pA}).

\begin{thm}[{\cite{MitM13B,Bar16pA}}]\label{thm:biharmonic:introduction}
Let $\Omega\subset\R^\dmn$ be a bounded Lipschitz domain with connected boundary. Let $-1/\pdmnMinusOne<\theta<1$. Then there is some $\kappa>0$ depending on~$\Omega$ and $\theta$ such that if $\dmn\geq 4$ and
\begin{align}\label{eqn:biharmonic:stripe}
0<\smooth<1,\quad 1< p< \infty, \quad
\frac{1}{2}-\frac{1}{\dmnMinusOne}-\kappa<\frac{1}{p}-\frac{\smooth}{\dmnMinusOne} < \frac{1}{2}+\kappa,
\end{align}
or if $\dmn=2$ or $\dmn=3$ and
\begin{align}
\label{eqn:biharmonic:stripe:2}
0<\smooth<1,\quad 1< p< \infty,\quad
0<\frac{1}{p}-\biggl(\frac{1-\kappa}{2}\biggr) \smooth < \frac{1+\kappa}{2},
\end{align}
then the biharmonic Dirichlet problem
\begin{equation}\label{eqn:Dirichlet:biharmonic}\left\{\begin{aligned} \Delta^2 u &= \Div_2\arr H \quad\text{in }\Omega,\\
\Tr_1^\Omega u &= \arr f ,\\
\doublebar{u}_{\dot W^{p,\smooth}_{2,av}(\Omega)} 
& \leq
C \doublebar{\arr f}_{\dot W\!A^p_{1,\smooth}(\partial\Omega)}
+C \doublebar{\arr H}_{L^{p,\smooth}_{av}(\Omega)}
\end{aligned}\right.
\end{equation}
and the biharmonic Neumann problem
\begin{equation}\label{eqn:Neumann:biharmonic}\left\{\begin{aligned} \Delta^2 u &= \Div_2 \arr H \quad\text{in }\Omega,\\
\M_{\mat A_\theta,\arr H}^\Omega u &= \arr g ,\\
\doublebar{u}_{\dot W^{p,\smooth}_{2,av}(\Omega)} 
& \leq
C \doublebar{\arr g}_{\dot N\!A^p_{1,\smooth-1}(\partial\Omega)}
+C \doublebar{\arr H}_{L^{p,\smooth}_{av}(\Omega)}
\end{aligned}\right.
\end{equation}
are both well posed.
\end{thm}
Here
$\mat A_\theta$ is the symmetric constant coefficient matrix such that 
\begin{equation}
\label{eqn:biharmonic:coefficients}
\langle \nabla^2 \psi(x),\mat A_\theta \nabla^2 \varphi(x)\rangle = 
\theta\, \overline{\Delta \psi(x)}\,\Delta\varphi(x)
+ (1-\theta) \sum_{j=1}^\dmn \sum_{k=1}^\dmn \overline{\partial_j \partial_k \psi(x)}\,\partial_j\partial_k \varphi(x).
\end{equation}

If $\dmn=2$ or $\dmn=3$, then 
Theorems~\ref{thm:atoms} and~\ref{thm:biharmonic:introduction} imply that the problems \eqref{eqn:Dirichlet:biharmonic} and~\eqref{eqn:Neumann:biharmonic} are well posed for all $p\leq 1$ and $\smooth$ satisfying the condition~\eqref{eqn:biharmonic:stripe} (and not only the stronger condition~\eqref{eqn:biharmonic:stripe:2}). As in Section~\ref{sec:MazMS10:introduction}, we may use 
\cite[Lemma~\ref*{P:lem:interpolation}]{Bar16pA} and \cite[Lemma~\ref*{P:lem:BVP:duality}]{Bar16pA} to show that if $\dmn=2$ or $\dmn=3$, then the problems \eqref{eqn:Dirichlet:biharmonic} and~\eqref{eqn:Neumann:biharmonic} are well posed for all $0<p\leq \infty$ and $\smooth$ satisfying the condition~\eqref{eqn:biharmonic:stripe}. That is, we have the following theorem.

\begin{thm}\label{thm:biharmonic:introduction:2}
Let $\Omega\subset\R^\dmn$ be a bounded Lipschitz domain with connected boundary and let $-1/\pdmnMinusOne<\theta<1$. Let $\kappa$ be as in Theorem~\ref{thm:biharmonic:introduction}.

Suppose that $p$ and $\smooth$ satisfy the condition~\eqref{eqn:biharmonic:stripe}. If $\dmn=2$, then the Dirichlet problem~\eqref{eqn:Dirichlet:biharmonic} and the Neumann problem~\eqref{eqn:Neumann:biharmonic} are well posed. If $\dmn=3$, the Neumann problem~\eqref{eqn:Neumann:biharmonic} is well posed.
\end{thm}

To the author's knowledge, the results of Theorem~\ref{thm:biharmonic:introduction:2} are new. Turning to the case of the three-dimensional Dirichlet problem, as mentioned above, well posedness of the Dirichlet problem
\begin{multline}\label{eqn:MitMW11}
\Delta^2 u=h\text{ in }\Omega\subset\R^3,\quad \Tr_1^\Omega u=\arr f, \\ \doublebar{u}_{F^{p,q}_{\smooth+1/p+1}(\Omega)}\leq  C\doublebar{h}_{F^{p,q}_{\smooth+1/p-3}(\Omega)} + C\doublebar{\arr f}_{W\!A^{p}_{1,\smooth}(\partial\Omega)}
\end{multline}
was established by I. Mitrea, M.~Mitrea and Wright in \cite{MitMW11} for  $p$ and $\smooth$ satisfying the condition~\eqref{eqn:biharmonic:stripe}. It is often possible to pass between the ${F^{p,q}_{\smooth+1/p+1}(\Omega)}$-norm and the  ${W^{p,\smooth}_{2,av}(\Omega)}$-norm; see, for example, \cite[Proposition~S]{AdoP98}. Thus, in the three-dimensional Dirichlet case Theorem~\ref{thm:atoms} yields no novel results.

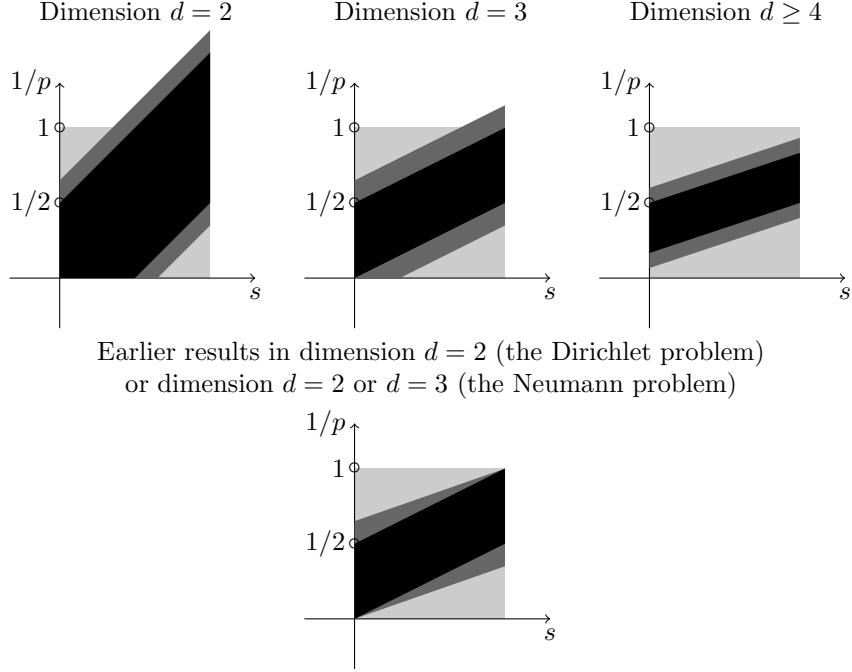
\begin{figure}
\begin{center}
\begin{tabular}{ccc}
Dimension $\dmn=2$ & Dimension $\dmn=3$ & Dimension $\dmn\geq 4$\\
\begin{tikzpicture}[scale=2]

\def\eps{0.3}
\def\enn{3}

\fill [white!80!black] (0,0) rectangle (1,1);

\draw [->] (-1/3,0)--(1.3,0) node [below] {$\smooth$};
\draw [->] (0,-1/3)--(0,1.3) node [left] {$1/p$};

\fill [white!40!black] 
	(0,0) --
	(1/2+\eps/2,0) --
	(1,1/2-\eps/2) -- 
	(1,3/2+\eps/2) --
	(0,1/2+\eps/2) --
	cycle;

\fill [black] 
	(0,0) 
	-- (1/2,0)
	-- (1,1/2) 
	-- (1,3/2) 
	-- (0,1/2) 
	-- cycle;

\node at (0,1) {$\circ$}; \node at (0,1) [left] {$1$};
\node at (0,1/2) {$\circ$}; \node at (0,1/2) [left] {$1/2$};

\end{tikzpicture}
&
\begin{tikzpicture}[scale=2]

\def\eps{0.3}
\def\enn{3}

\fill [white!80!black] (0,0) rectangle (1,1);

\draw [->] (-1/3,0)--(1.3,0) node [below] {$\smooth$};
\draw [->] (0,-1/3)--(0,1.3) node [left] {$1/p$};

\fill [white!40!black] 
	(0,0) --
	(\eps,0) --
	(1,1/2-\eps/2) -- 
	(1,1+\eps/2) --
	(0,1/2+\eps/2) --
	cycle;

\fill [black] 
	(0,0)
	-- (1,1/2) 
	-- (1,1) 
	-- (0,1/2) 
	-- cycle;

\node at (0,1) {$\circ$}; \node at (0,1) [left] {$1$};
\node at (0,1/2) {$\circ$}; \node at (0,1/2) [left] {$1/2$};

\end{tikzpicture}
&
\begin{tikzpicture}[scale=2]

\def\eps{0.1}
\def\enn{3}

\fill [white!80!black] (0,0) rectangle (1,1);

\draw [->] (-1/3,0)--(1.3,0) node [below] {$\smooth$};
\draw [->] (0,-1/3)--(0,1.3) node [left] {$1/p$};

\fill [white!40!black] (0,1/2)  
	-- (0,1/2+\eps) 
	-- (1,1/2+1/\enn+\eps) 
	-- (1,1/2-\eps) 
	-- (0,1/2-1/\enn-\eps) 
	;

\fill [black] 
	(0,1/2)  
	-- (0,1/2) 
	-- (1,1/2+1/\enn) 
	-- (1,1/2) 
	-- (0,1/2-1/\enn) 
	;

\node at (0,1) {$\circ$}; \node at (0,1) [left] {$1$};
\node at (0,1/2) {$\circ$}; \node at (0,1/2) [left] {$1/2$};

\end{tikzpicture}
\end{tabular}

Earlier results in dimension $\dmn=2$ (the Dirichlet problem) \\ or dimension $\dmn=2$ or $\dmn=3$ (the Neumann problem)

\begin{tikzpicture}[scale=2]

\def\eps{0.3}
\def\enn{3}

\fill [white!80!black] (0,0) rectangle (1,1);

\draw [->] (-1/3,0)--(1.3,0) node [below] {$\smooth$};
\draw [->] (0,-1/3)--(0,1.3) node [left] {$1/p$};

\fill [white!40!black] 
	(0,0) --
	(0,0) --
	(1,1/2-\eps/2) -- 
	(1,1) --
	(0,1/2+\eps/2) --
	cycle;

\fill [black] 
	(0,0) 
	-- (1,1/2) 
	-- (1,1) 
	-- (0,1/2) 
	-- cycle;

\node at (0,1) {$\circ$}; \node at (0,1) [left] {$1$};
\node at (0,1/2) {$\circ$}; \node at (0,1/2) [left] {$1/2$};

\end{tikzpicture}

\end{center}
\caption{The Dirichlet and Neumann problems \eqref{eqn:Dirichlet:biharmonic} and \eqref{eqn:Neumann:biharmonic} for the biharmonic operator $\Delta^2$ are well posed whenever the point $(\smooth,1/p)$ lies in the black or dark gray regions. The black region depends only on the dimension $\dmn$; the gray regions also depend on~$\Omega$. If $\dmn\geq3$ (the Dirichlet problem), if $\dmn\geq 4$ (the Neumann problem), or if $(\smooth,1/p)$ lies in the region shown below, then the result is known from \cite{AdoP98,MitMW11,MitM13B}; the extended range in the case $\dmn=2$ or $\dmn=3$ is due to Theorem~\ref{thm:atoms} in the present paper.} \label{fig:MitM13B}\end{figure}

We mention one final class of new well posedness results.
In \cite[Theorem~\ref*{P:thm:biharmonic:perturb}]{Bar16pA}, we established well posedness results for operators close to the biharmonic operator for $p$, $\smooth$ as in Theorem~\ref{thm:biharmonic:introduction}. The same technique allows us to establish perturbative results for $p$, $\smooth$ as in Theorem~\ref{thm:biharmonic:introduction:2}.
\begin{thm}\label{thm:biharmonic:perturb}
Let $N\geq 1$ be an integer, and for each $1\leq j\leq N$, let $\theta_j\in \R$; in the case of the Neumann problem we additionally require $-1/\pdmnMinusOne<\theta_j<1$. Let $\Omega\subset\R^2$ or $\Omega\subset\R^3$ be a bounded simply connected Lipschitz domain,  and let $\kappa_j$ be as in Theorem~\ref{thm:biharmonic:introduction}. Let $0<\delta<\kappa=\min_j\kappa_j$. 

Let $L$ be an operator of the form \eqref{eqn:divergence}, with $m=2$ and defined in the weak sense of formula~\eqref{eqn:weak}, associated to coefficients~$\mat A$.
Then there is some $\varepsilon>0$ such that,  if
\begin{equation*}\sup_{j,\alpha,\beta,x}\abs{A^{jj}_{\alpha\beta}(x)-(A_{\theta_j})_{\alpha\beta}} +\sup_{\substack{{j,k,\alpha,\beta,x}\\j\neq k}}\abs{A^{jk}_{\alpha\beta}(x)}<\varepsilon\end{equation*}
then the Dirichlet problem~\eqref{eqn:Dirichlet:p:smooth}
and the Neumann problem~\eqref{eqn:Neumann:p:smooth}, with $m=2$,
are well posed whenever 
\begin{gather*}
\delta\leq\smooth\leq 1-\delta,\quad
0<p\leq\infty,\quad
\frac{1}{2}-\frac{1}{\dmnMinusOne}-(\kappa-\delta) \leq \frac{1}{p}-\frac{\smooth}{\dmnMinusOne} \leq \frac{1}{2}+(\kappa-\delta)
.\end{gather*}
\end{thm}

\subsection{Outline of the paper}

The paper is organized as follows. We will prove Theorem~\ref{thm:atoms} in Sections~\ref{sec:preliminaries}--\ref{sec:proof}. In Section~\ref{sec:preliminaries} we will begin the proof of Theorem~\ref{thm:atoms}; we will establish uniqueness of solutions, provide some preliminary arguments, and will construct a solution to $L\vec u=\Div_m\arr\Phi$ in~$\Omega$ provided $\arr\Phi$ is supported in a Whitney ball. In Section~\ref{sec:atoms:boundary} we will bound~$\vec u$; it is this section that contains most of the technical arguments of the paper. In Section~\ref{sec:proof} we will pass from data $\arr \Phi$ supported in a Whitney ball to data $\arr H$ supported in all of~$\Omega$. Finally, in Section~\ref{sec:A0}, we will resolve some differences between well posedness results as stated in the literature, and the well posedness results required by Theorem~\ref{thm:atoms}; the results of Section~\ref{sec:A0} were used in Section~\ref{sec:MazMS10:introduction} above.

\subsection*{Acknowledgements}

The author would like to thank Steve Hofmann and Svitlana Mayboroda for many useful discussions concerning the theory of higher-order elliptic equations. The author would also like to thank the American Institute of Mathematics for hosting the SQuaRE workshop on ``Singular integral operators and solvability of boundary problems for elliptic equations with rough coefficients,'' at which many of these discussions occurred, and the Mathematical Sciences Research Institute for hosting a Program on Harmonic Analysis, during which substantial parts of this paper were written.

\section{Preliminaries}\label{sec:preliminaries}

We will take our notation for multiindices and our definitions of function spaces and Lipschitz domains from \cite{Bar16pB},  and our notions of elliptic operators and well posedness from \cite{Bar16pA}. 
We remark that throughout the paper, $C$ will denote a positive constant whose value may change from line to line. We say that $A\approx B$ if $A\leq CB$ and $B\leq CA$.

There remains to define the primary Lipschitz constant mentioned in Theorem~\ref{thm:atoms}. 

\begin{dfn}\label{dfn:Lipschitz}
If $\Omega$ is a Lipschitz domain, as defined by \cite[Definition~2.2]{Bar16pB}, let $(M,n,c_0)$ be the Lipschitz character of $\Omega$ given therein. We refer to $M$ as the \emph{primary Lipschitz constant} of~$\Omega$.
\end{dfn}

In this section we will begin the proof of Theorem~\ref{thm:atoms}.
We begin by establishing uniqueness of solutions. This follows from well posedness of the problems \eqref{eqn:Dirichlet:q-}, \eqref{eqn:Dirichlet:tilde} or \eqref{eqn:Neumann:q-} and certain embedding and interpolation results of \cite{Bar16pA}.

\begin{lem}\label{lem:unique:atoms} Suppose that the conditions of Theorem~\ref{thm:atoms} are valid. Then for each $\arr H\in L^{p,\smooth}_{av}(\Omega)$ and each $\arr f\in \dot W\!A^p_{m-1,\smooth}(\partial\Omega)$ or $\arr g\in \dot N\!A^p_{m-1,\smooth-1}(\partial\Omega)$, there is at most one solution to the problem~\eqref{eqn:Dirichlet:p:smooth} or~\eqref{eqn:Neumann:p:smooth}.
\end{lem}
\begin{proof}
If $\Omega$ is bounded, then the numbers $r=p$, $\omega=\smooth$, $q=q_-$ and $\sigma=\sigma_-$ satisfy the conditions of \cite[Lemma~\ref*{P:lem:embedding}]{Bar16pA}. Thus, by \cite[Corollary~\ref*{P:cor:unique:extrapolate}]{Bar16pA}, solutions to the problems \eqref{eqn:Dirichlet:p:smooth} and~\eqref{eqn:Neumann:p:smooth} are unique.

Otherwise, let $q$ and $\sigma$ satisfy $\pdmnMinusOne/q-\sigma=\pdmnMinusOne/p-\smooth$ and $1/q=\theta/q_++(1-\theta)/q_-$, $\sigma=\theta\sigma_++(1-\theta)\sigma_-$ for some $\theta\in\R$. 
See Figure~\ref{fig:twolines}.
By the given bounds on $q_\pm$ and $\sigma_\pm$, we have that $0<\sigma<\smooth$ and $1<q<\infty$, and furthermore $0<\theta<1$.

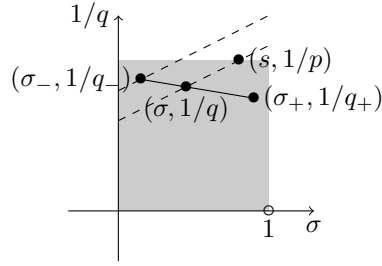
\begin{figure}
\begin{center}
\begin{tikzpicture}[scale=2]

\def\enn{2}

\fill [white!80!black] (0,0) -- (1,0) -- (1,1) -- (0,1) -- cycle;

\draw [->] (-1/3,0)--(1.3,0) node [below] {$\sigma$};
\draw [->] (0,-1/3)--(0,1.3) node [left] {$1/q$};
\node at (1,0) {$\circ$}; \node at (1,0) [below] {$1$};

\def  \YI{0.8}
\def \YII{0.6}
\def\YIII{0.3}

\def\sigmaminus{0.15}
\def\sm{0.8}
\def\sigmaplus{0.9}

\draw [dashed, name path=line 1] (0,\YII) -- +(1,1/\enn);
\draw [dashed] (0,\YI) -- +(1,1/\enn);

\path (\sm,\YII+\sm/\enn) node {$\bullet$} node [right] {$(\smooth,1/p)$}; 
\path (\sigmaminus,\YI+\sigmaminus/\enn) node {$\bullet$} node [left] {$(\sigma_-,1/q_-)$}; 
\path (\sigmaplus,\YIII+\sigmaplus/\enn) node {$\bullet$} node [right] {$(\sigma_+,1/q_+)$}; 

\draw[name path=line 2] (\sigmaminus,\YI+\sigmaminus/\enn) -- (\sigmaplus,\YIII+\sigmaplus/\enn);

\path  [name intersections={of=line 1 and line 2}] (intersection-1) node {$\bullet$} node [below] {$(\sigma,1/q)$};

\end{tikzpicture}
\end{center}
\caption{Lemma~\ref{lem:unique:atoms}: the point $(\sigma,1/q)$ lies on the line through $(\sigma_-,1/q_-)$ and $(\sigma_+,1/q_+)$, and on the line $\sigma-\pdmnMinusOne(1/q)=\smooth-\pdmnMinusOne(1/p)$. The quantity $\sigma-\pdmnMinusOne/q$ is constant on each of the dashed lines. }
\label{fig:twolines}
\end{figure}

By the interpolation result \cite[Lemma~\ref*{P:lem:interpolation}]{Bar16pA}, we have well posedness of the Dirichlet or Neumann problem \eqref{eqn:Dirichlet:p:smooth} or~\eqref{eqn:Neumann:p:smooth} with $p=q$, $\smooth=\sigma$. Observe that $q$, $\sigma$, and $\omega=\smooth$, $r=p$ satisfy the conditions of \cite[Lemma~\ref*{P:lem:embedding}]{Bar16pA} even if $\Omega$ is unbounded. The conclusion follows from \cite[Corollary~\ref*{P:cor:unique:extrapolate}]{Bar16pA}.
\end{proof}



We now recall the following theorem; this theorem will be useful throughout the paper.
In the interior case $\Omega=B(x,2r)$, the result may be found in \cite{Cam80,AusQ00,Bar16}. The case of Dirichlet boundary values was established in \cite{Bar16}, while the case of Neumann boundary values may be established using \cite[Lemma~5.3]{Tol16p} and the proof of \cite[Theorem~36]{Bar16}. 
\begin{thm}
\label{thm:Meyers} 
Let $L$ be an operator of the form \eqref{eqn:divergence}, defined in the weak sense of formula~\eqref{eqn:weak},  of order $2m$ and associated to coefficients~$\mat A$ that satisfy the bound~\eqref{eqn:elliptic:bounded} and the ellipticity condition~\eqref{eqn:elliptic:everywhere}. Suppose that $L\vec u=\Div_m\arr H$ in $B(x,2r)$ for some $x\in\R^\dmn$ and some $r>0$. If $0\leq j\leq m$, $0<p<\infty$, and $0<q\leq 2$, then
\begin{align*}
\biggl(\int_{B(x,r)}\abs{\nabla^j   u}^q\biggr)^{1/q}
&\leq 
	\frac{C}{r^{\pdmn/p-\pdmn/q}}
	\biggl(\int_{B(x,2r)}\abs{\nabla^j   u}^p\biggr)^{1/p} 
	+ C r^{m-j}\biggl(\int_{B(x,2r)}\abs{\arr H}^q\biggr)^{1/q}
\end{align*}
for some constant $C$ depending only on $p$, $q$, the dimension $\dmn$, the constants $\lambda$ and $\Lambda$ in the bounds \eqref{eqn:elliptic:bounded} and~\eqref{eqn:elliptic:everywhere}, and the Lipschitz character of~$\Omega$.

Let $\Omega\subset\R^\dmn$ be a Lipschitz domain, let $x\in\partial\Omega$, and let $\vec u\in \dot W^2_m(B(x,2r))$ with $L\vec u=0$ in $\Omega\cap B(x,2r)$.
Suppose that either $\mat A$ satisfies the ellipticity condition~\eqref{eqn:elliptic:everywhere} and $\Tr_{k}^\Omega\vec u=0$ on~$\partial\Omega\cap B(x,2r)$ for all $0\leq k\leq m-1$, or $\mat A$ satisfies the ellipticity condition~\eqref{eqn:elliptic:domain} and $\M_{\mat A,0}^\Omega \vec u=0$ on~$\partial\Omega\cap B(x,2r)$.

Then  if $0\leq j\leq m$, $0<p<\infty$, and $0<q\leq 2$, then
\begin{align*}
\biggl(\int_{B(x,r)\cap\Omega
}\abs{\nabla^j   u}^q\biggr)^{1/q}
&\leq 
	\frac{C}{r^{\pdmn/p-\pdmn/q}}
	\biggl(\int_{B(x,2r)\cap\Omega}\abs{\nabla^j   u}^p\biggr)^{1/p} 
.\end{align*}
\end{thm}

As a consequence we have the following equivalence between the averaged norms $\dot W^{q,\sigma}_{m,av}(T)$ and unaveraged norms.

\begin{lem}\label{lem:unaveraged} Let $T$ be an open set. If $0<q\leq 2$ and $\sigma\in\R$, then 
\begin{equation}
\label{eqn:unaveraged<averaged}
\int_T \abs{\arr\Psi(x)}^q \dist(x,\partial T)^{q-1-q\sigma}\,dx
\leq C\doublebar{\arr\Psi}_{L^{q,\sigma}_{av}(T)}^q
\end{equation}
for all $\arr\Psi\in {L^{q,\sigma}_{av}(T)}$, where $C$ depends only on~$q$ and~$\sigma$.

Conversely, suppose that $\vec u\in \dot W^2_{m,loc}(T)$ (that is, suppose that $\nabla^m\vec u$ is locally square-integrable in~$T$). Let $0<q\leq \infty$ and let $\sigma\in\R$. Let $L$ be an operator of the form \eqref{eqn:divergence}, defined in the weak sense of formula~\eqref{eqn:weak}, associated to coefficients $\mat A$ that satisfy the bounds~\eqref{eqn:elliptic:bounded} and~\eqref{eqn:elliptic:everywhere}. Suppose that $L\vec u=\Div_m\arr H$ in~$T$ for some $\arr H\in L^{q,\sigma}_{av}(\Omega)$. Then
\begin{align*}\doublebar{\vec u}_{\dot W^{q,\sigma}_{m,av}(T)}^q
&=\int_T \biggl(\fint_{B(x,\dist(x,\partial T)/2)} \abs{\nabla^m\vec u}^2\biggr)^{q/2} \dist(x,\partial T)^{q-1-q\sigma}\,dx 
\\&\leq C
\int_T \abs{\nabla^m\vec u(x)}^q \dist(x,\partial T)^{q-1-q\sigma}\,dx
+C\doublebar{\arr H}_{L^{q,\sigma}_{av}(T)}^q
\end{align*}
for some constant $C$ depending only on $q$, $\sigma$ and the quantities mentioned in Theorem~\ref{thm:Meyers}.
\end{lem}

\begin{proof}
It is straightforward to establish that
\begin{equation*}
\int_T \abs{\arr\Psi(x)}^q \dist(x,\partial T)^{q-1-q\sigma}\,dx
\approx
\int_T \biggl(\fint_{B(x,a\dist(x,\partial T))} \abs{\arr\Psi}^q\biggr) \dist(x,\partial T)^{q-1-q\sigma}\,dx 
\end{equation*}
for any $0<a<1$, with comparability constants depending only on~$a$, $q$ and~$\sigma$. By H\"older's inequality, if $q\leq 2$ then the inequality~\eqref{eqn:unaveraged<averaged} is valid.

Conversely, by Theorem~\ref{thm:Meyers}, 
\begin{multline*}\biggl(\fint_{B(x,\dist(x,\partial T)/2)} \abs{\nabla^m\vec u}^2\biggr)^{1/2}
\leq C\biggl(\fint_{B(x,(3/4)\dist(x,\partial T))} \abs{\nabla^m\vec u}^q\biggr)^{1/q}
\\+C\biggl(\fint_{B(x,(3/4)\dist(x,\partial T))} \abs{\arr H}^2\biggr)^{1/2}.\end{multline*}
It may easily be shown (see \cite[Section~3]{Bar16pB}) that
\begin{equation*}\doublebar{\arr H}_{L^{q,\sigma}_{av}(T)}^q\approx\int_T \biggl(\fint_{B(x,a\dist(x,\partial T))} \abs{\arr H}^2\biggr)^{q/2} \dist(x,\partial T)^{q-1-q\sigma}\,dx \end{equation*}
for any $0<a<1$,
and so the proof is complete.
\end{proof}

We now establish existence of solutions. By \cite[Lemma~\ref*{P:lem:zero:boundary}]{Bar16pA}, we need only consider the case of homogeneous boundary data $\arr f=0$ or~$\arr g=0$.

We begin by considering boundary value problems for the differential equation $L\vec u=\Div_m\arr\Phi$, where $\arr\Phi$ is supported in a Whitney ball.

\begin{lem}\label{lem:atoms:-}
Let $L$, $\Omega$, $p$, $\smooth$, $q_-$, and $\sigma_-$ satisfy the conditions of Theorem~\ref{thm:atoms}. Let $\arr\Phi\in L^2(\Omega)$ and suppose that $\supp \arr\Phi\subset B(x_0,\dist(x_0,\partial\Omega)/2)$ for some $x_0\in\Omega$.

Let $\vec u \in {\dot W^{q_-,\sigma_-}_{m,av}(\Omega)}$ be the solution to the problem \eqref{eqn:Dirichlet:q-} or~\eqref{eqn:Neumann:q-}. Let $0<h<1$, and let $\Gamma(x_0)=\Gamma_h(x_0)=\{x\in\Omega: \abs{x_0-x}<\frac{1}{h}\dist(x,\partial\Omega)\}$.

Then $\vec u$ satisfies the estimate 
\begin{equation*}\int_{\Gamma(x_0)} \abs{\nabla^m\vec u(x)}^p\dist(x,\partial\Omega)^{p-1-p\smooth}\,dx
\leq 
C\doublebar{\arr\Phi}_{L^{p,\smooth}_{av}(\Omega)}^p
\end{equation*}
where $C$ depends on $h$, $p$, $\smooth$, $q_-$, $\sigma_-$ and the standard parameters.
\end{lem}

\begin{proof} 
We have that $p\leq 2$. Thus, by the bound~\eqref{eqn:unaveraged<averaged},
\begin{equation*}\int_{\Gamma(x_0)} \abs{\nabla^m\vec u(x)}^p\dist(x,\partial\Omega)^{p-1-p\smooth}\,dx
\leq \doublebar{\1_{\Gamma(x_0)} \nabla^m \vec u}_{L^{p,\smooth}_{av}(\Omega)}^p
.\end{equation*}

If $x\in \Omega$ and $B(x,\dist(x,\partial\Omega)/2)\cap \Gamma(x_0)\neq \emptyset$, then there is some $y$ such that
\begin{align*}\abs{x-x_0}&<\abs{x_0-y}+\frac{1}{2}\dist(x,\partial\Omega)
<
	\frac{1}{h}\dist(y,\partial\Omega) +\frac{1}{2}\dist(x,\partial\Omega)
\\&\leq
	\frac{1}{h}(\dist(x,\partial\Omega)+\abs{x-y}) +\frac{1}{2}\dist(x,\partial\Omega)
<\frac{3+h}{2h}\dist(x,\partial\Omega)
\end{align*}
and so $x\in\widetilde \Gamma(x_0)=\Gamma_{2h/(3+h)}(x_0)$.

Because $q_->p$, we may use H\"older's inequality to see that
\begin{multline*}
\doublebar{\1_{\Gamma(x_0)} \nabla^m \vec u}_{L^{p,\smooth}_{av}(\Omega)}^p
\\\leq
\biggl(\int_{\widetilde\Gamma(x_0)} \biggl(\fint_{B(x,\dist(x,\partial\Omega)/2)}\abs{\nabla^m\vec u}^2\biggr)^{q_-/2}\dist(x,\partial\Omega)^{q_--1-q_-\sigma_-}\, dx 
\biggr)^{p/q_-}
\\\times
\biggl(
\int_{\widetilde\Gamma(x_0)}\dist(x,\partial\Omega)^{(\sigma_--\smooth)pq_-/(q_--p)-1}\, dx 
\biggr)^{1-p/q_-}
.\end{multline*}
The exponent ${(\sigma_--\smooth)pq_-/(q_--p)-1}$ is negative, and in particular is less than $-\pdmn$. Applying the definition of $\widetilde\Gamma(x_0)$ yields that (for $\delta=\dist(x_0,\partial\Omega)$),
{\multlinegap=0pt\begin{multline*}
\int_{\widetilde\Gamma(x_0)}\dist(x,\partial\Omega)^{(\sigma_--\smooth)pq_-/(q_--p)-1}\, dx 
\\\leq
C\delta^{\dmnMinusOne+(\sigma_--\smooth)pq_-/(q_--p)}+
C\int_{\R^\dmn\setminus B(x_0,\delta/2)} (h\abs{x-x_0})^{(\sigma_--\smooth)pq_-/(q_--p)-1}\, dx 
\\\leq 
C\delta^\dmn(h\delta)^{(\sigma_--\smooth)pq_-/(q_--p)-1}
.\end{multline*}}%

By assumption, $\doublebar{\vec u}_{\dot W^{q_-,\sigma_-}_{m,av}(\Omega)}\leq C\doublebar{\arr\Phi}_{L^{q_-,\sigma_-}_{av}(\Omega)}$.
But because $\arr\Phi$ is supported in $B(x_0,\delta/2)$, we have that
\begin{equation}
\label{eqn:whitney:q:sigma}
\doublebar{\arr\Phi}_{L^{q,\sigma}_{av}(\Omega)}\approx \delta^{1-\pdmn/2+\pdmnMinusOne/q-\sigma}
\doublebar{\arr\Phi}_{L^2(B(x_0,\delta/2))}
\end{equation}
for any $\sigma\in\R$ and any $0<q\leq \infty$, and so
\begin{equation*}
\int_{\Gamma(x_0)} \abs{\nabla^m\vec u(x)}^p\dist(x,\partial\Omega)^{p-1-p\smooth}\, dx 
\leq C(h)\doublebar{\arr\Phi}_{L^{p,\smooth}_{av}(\Omega)}^p
\end{equation*}
as desired.
\end{proof}

We now come to the remaining region.
\begin{lem}\label{lem:atoms:boundary}
Suppose that the conditions of Theorem~\ref{thm:atoms} are valid and let $p$, $\smooth$ be as in Theorem~\ref{thm:atoms}.
Let $\arr\Phi$ and $\vec u$ be as in Lemma~\ref{lem:atoms:-}.

If $h>0$ is small enough (depending only on the Lipschitz character of~$\Omega$ and not on the choice of~$x_0$), then
\begin{equation}\label{eqn:atoms:boundary}
\int_{\Omega\setminus\Gamma(x_0)} \abs{\nabla^m\vec u(x)}^p\dist(x,\partial\Omega)^{p-1-p\smooth}\, dx 
\leq C(h)\doublebar{\arr\Phi}_{L^{p,\smooth}_{av}(\Omega)}^p
.\end{equation}
\end{lem}
We will prove this lemma in the next section, and complete the proof of Theorem~\ref{thm:atoms} in Section~\ref{sec:proof}.

\section{Proof of Lemma~\ref{lem:atoms:boundary}}
\label{sec:atoms:boundary}


We will begin (Section~\ref{sec:far}) by treating the case where $x_0$ is far from~$\partial\Omega$; we will treat the more useful but much more intricate case of $x_0$ near $\partial\Omega$ in Section~\ref{sec:near}. 

Throughout this section we will let $\delta=\dist(x_0,\partial\Omega)$.

\subsection{The case of $x_0$ far from~$\partial\Omega$}\label{sec:far}

If $\partial\Omega$ is bounded but $\Omega$ is not, we must consider the case where $\delta=\dist(x_0,\partial\Omega)\geq \diam\partial\Omega$.

Recall that if $x\in \Omega \setminus \Gamma(x_0)$, then $\dist(x,\partial\Omega)<h\abs{x-x_0}$. An elementary argument involving the triangle inequality yields that $\abs{x-x_0} < \frac{1}{1-h}(\diam(\partial\Omega)+\dist(x_0,\partial\Omega))$. Thus, 
\begin{equation}\label{eqn:Omega:Gamma:ball}
\Omega\subset \Gamma(x_0)\cup B\Bigl(x_0,\frac{\diam\Omega+\dist(x_0,\partial\Omega)}{1-h}\Bigr).
\end{equation}
Thus, we wish to bound $\vec u$ in a ball.

We begin with the following lemma.
\begin{lem}\label{lem:embedding:local}
Let $\Omega\subset\R^\dmn$ be an open set, let
$\smooth<\sigma$, and let $0<p\leq 2$.
Suppose that $0<q\leq \infty$ is such that 
$\smooth-\frac{\dmnMinusOne}{p}  \leq \sigma-\frac{\dmnMinusOne}{ q}$.
Let $x_0\in \overline \Omega$ and let $R>\dist(x_0,\partial\Omega)$.
Then 
\begin{equation*}\int_{\Omega \cap B(x_0,R)} \abs{\arr \Psi(x)}^p\dist(x,\partial\Omega)^{p-1-p\smooth}\,dx
\leq C R^{\pdmnMinusOne(1-p/q)+p\sigma-p\smooth} \doublebar{\arr\Psi}_{L^{q,\sigma}_{av}(\Omega)}^p.
\end{equation*}
\end{lem}

\begin{proof}
If $x\in\Omega$, let $B(x,\Omega)={B(x,\dist(x,\partial\Omega)/2)}$. Let $V=\Omega\cap B(x_0,7R)$; observe that if $x\in B(x_0,R)$, then $\dist(x,\partial\Omega) = \dist(x,\partial V)$.
Thus,
\begin{equation*}\int_{\Omega \cap B(x_0,R)} \abs{\arr\Psi(x)}^p\dist(x,\partial\Omega)^{p-1-p\smooth}\,dx
=\int_{V} \abs{\1_{B(x_0,R)}\arr\Psi(x)}^p\dist(x,\partial V)^{p-1-p\smooth}\,dx
.\end{equation*}
By the bound~\eqref{eqn:unaveraged<averaged}, and because $p\leq 2$,
\begin{equation*}
\int_{V} \abs{\1_{B(x_0,R)}\arr\Psi(x)}^p\dist(x,\partial V)^{p-1-p\smooth}\,dx
\leq
C\doublebar{\1_{B(x_0,R)}\arr\Psi}_{L^{p,\smooth}_{av}(V)}^p
.\end{equation*}
By \cite[Lemma~\ref*{P:lem:embedding}]{Bar16pA}, and because $\diam V\leq 14R$, we have that 
\begin{equation*}\doublebar{\1_{B(x_0,R)}\arr\Psi}_{L^{p,\smooth}_{av}(V)}
\leq
CR^{\pdmnMinusOne(1/p-1 /q) + \sigma -\smooth}\doublebar{\1_{B(x_0,R)}\arr\Psi}_{L^{q,\sigma}_{av}(V)}
.\end{equation*}
If $x\in \Omega$ and $B(x,\dist(x,\partial\Omega)/2)\cap B(x_0,R)\neq \emptyset$, then we may show using the triangle inequality that $\dist(x,\partial\Omega)\leq \dist(x,\partial B(x_0,7R))$ and so $\dist(x,\partial\Omega)=\dist(x,\partial V)$. Thus,
\begin{equation*}\doublebar{\1_{B(x_0,R)}\arr\Psi}_{L^{q,\sigma}_{av}(V)} = \doublebar{\1_{B(x_0,R)}\arr\Psi}_{L^{q,\sigma}_{av}(\Omega)}
\leq \doublebar{\arr\Psi}_{L^{q,\sigma}_{av}(\Omega)}\end{equation*}
as desired.
\end{proof}

Let $\vec u_+ \in {\dot W^{q_+,\sigma_+}_{m,av}(\Omega)}$ be the solution to the problem \eqref{eqn:Dirichlet:q+} or~\eqref{eqn:Neumann:q-}. Because $\Omega$ is unbounded, by assumption $\vec u_+=\vec u$. By Lemma~\ref{lem:embedding:local} with $\arr\Psi=\nabla^m\vec u$, $q=q_+$ and $\sigma=\sigma_+$, we have that
\begin{multline*}\int_{\Omega \cap B(x_0,2\delta/(1+h))} \abs{\nabla^m\vec u(x)}^p\dist(x,\partial\Omega)^{p-1-p\smooth}\,dx
\\ \leq C(h) \,\delta^{\pdmnMinusOne(1-p/q)+p\sigma-p\smooth} \doublebar{\vec u}_{\dot W^{q_+,\sigma_+}_{m,av}(\Omega)}^p.
\end{multline*}
By assumption $\doublebar{\vec u}_{\dot W^{q_+,\sigma_+}_{m,av}(\Omega)}\leq \doublebar{\arr\Phi}_{L^{q_+,\sigma_+}_{av}(\Omega)}$, and so by the bound~\eqref{eqn:whitney:q:sigma} we have that
\begin{equation*}\int_{\Omega \cap B(x_0,2\delta/(1-h))} \abs{\nabla^m\vec u(x)}^p\dist(x,\partial\Omega)^{p-1-p\smooth}\,dx
\leq C(h) \doublebar{\arr\Phi}_{L^{p,\smooth}_{av}(\Omega)}^p.
\end{equation*}
By Lemma~\ref{lem:atoms:-} and the inclusion~\eqref{eqn:Omega:Gamma:ball}, we have that Lemma~\ref{lem:atoms:boundary} is valid whenever $\dist(x_0,\partial\Omega)\geq \diam\partial\Omega$.


\subsection{The case of $x_0$ near~$\partial\Omega$}\label{sec:near}

Throughout this section we will assume that $\delta=\dist(x_0,\partial\Omega)<\diam\partial\Omega$. (If $\Omega$ is a Lipschitz graph domain then $\diam\partial\Omega=\infty$ and so this is true for all~$x_0\in\R^\dmn$; if $\Omega$ is bounded then this is true for all $x_0\in\Omega$.)

Let $\Xi=\Omega \setminus \Gamma(x_0)$.
Observe that $\Xi$ is a set of points lying near $\partial\Omega$ and far from~$x_0$. We wish to cover $\Xi$ by simply connected regions $T\subset\Omega$ with primary Lipschitz constant~$M_0$, and then use well posedness of the problem~\eqref{eqn:Dirichlet:tilde} or~\eqref{eqn:Neumann:tilde} in~$T$ to bound $\vec u$ in~$T$.

We define the regions $T$ as follows. 
\begin{dfn}\label{dfn:tent} 
Recall from Definition~\ref{dfn:Lipschitz} and \cite[Definition~2.2]{Bar16pB} that $V$ is a Lipschitz graph domain with primary Lipschitz constant~$M$ if there is some system of coordinates and some Lipschitz function $\psi:\R^\dmnMinusOne\mapsto \R$ with $\doublebar{\nabla\psi}_{L^\infty(\R^\dmnMinusOne)}\leq M$ and with $V=\{(x',t):x'\in\R^\dmnMinusOne,\>t>\psi(x')\}$. Furthermore, if $\Omega$ is a Lipschitz domain, then either $\Omega=V$ for some Lipschitz graph domain~$V$ or there is some  $c_0>0$ and some $r_\Omega>0$ such that, if $z\in\partial\Omega$, then there is a Lipschitz graph domain $V$ such that $B(z,r_\Omega/c_0)\cap \Omega = B(z,r_\Omega/c_0)\cap V$.

Choose some $z\in\partial\Omega$. In the coordinates associated with the Lipschitz graph domain~$V$, we write $z=(z',t_0)$.

If $\rho>0$, then let $K(z,\rho)$ be the cone $\{(y',t): M_0\abs{y'-z'} < t-t_0+M_0\rho\}$ with vertex at $(z',t_0-M_0\rho)$ and with slope $M_0$. Because $\doublebar{\nabla\psi}_{L^\infty(\R^\dmnMinusOne)}\leq M$, we have that
\begin{equation*}V\supset \{(y',t):t> t_0+M\abs{y'-z'}\}\end{equation*}
and it is straightforward to establish that if $M<M_0$ then
\begin{equation*}\{(y',t):t> t_0+M\abs{y'-z'}\}\supset K(z,\rho)\cap\{(y',t):t>t_0+MM_0\rho/(M_0-M)\}.\end{equation*}
Let $B$ be the ball with center on the axis of~$K$ such that $\partial B$ is tangent to $\partial K$ and such that 
\begin{equation*}\partial B\cap\partial K = \{(y',t_0+MM_0\rho/(M_0-M)+\rho):\abs{y'-z'} =M_0\rho/(M_0-M)+\rho/M_0\}.\end{equation*} 

We let 
\begin{gather}
\widetilde K(z,\rho)=K(z,\rho)\cap \{(y',t):t<t_0+MM_0\rho/(M_0-M)+\rho\},
\nonumber\\\label{eqn:T}
\widetilde T(z,\rho) = \widetilde K(z,\rho)\cup B,\qquad T(z,\rho) = \widetilde T(z,\rho)\cap V. \end{gather}
See Figure~\ref{fig:tent}.
\end{dfn}

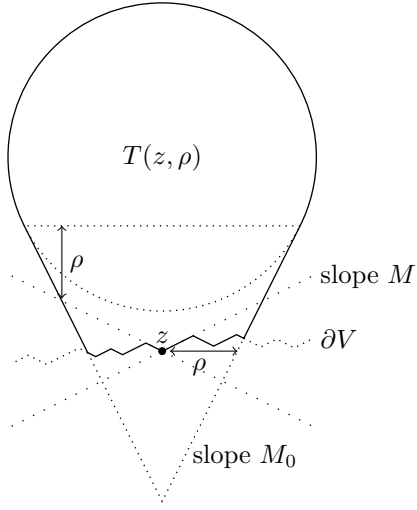
\begin{figure}
\begin{center}
\begin{tikzpicture}
\draw  [dotted, line width=0.50188pt] (-0.99025,-0.0195)--(-1.10847,0.0396)--(-1.53877,-0.17554)--(-1.77966,-0.0551)--(-2,-0.16525);
\draw  [dotted, line width=0.50188pt] (1.08513,0.1703)--(1.31395,0.0559)--(1.52322,0.16052)--(1.75714,0.04356)--(1.99997,0.16498);
\draw  [loosely dotted, line width=0.50188pt] (-2,1)--(2,-1);
\draw  [loosely dotted, line width=0.50188pt] (-2,-1)--(2,1);
\draw  [dotted, line width=0.50188pt] (-1.83656,1.67319) ..controls (-1.73854,1.47543) and (-1.60835,1.293).. (-1.44934,1.13399) ..controls (-1.06496,0.7496) and (-0.54364,0.5336).. (0,0.5336) ..controls (0.54364,0.5336) and (1.06496,0.7496).. (1.44934,1.13399) ..controls (1.60497,1.2896) and (1.73297,1.46765).. (1.83023,1.66054);
\draw  [dotted, line width=0.50188pt] (-0.99023,-0.01953)--(0,-2)--(1.08514,0.1703);
\draw  [dotted, line width=0.50188pt] (1.83333,1.66667)--(-1.83333,1.66667);
\draw  [line width=0.50188pt] (-0.99025,-0.0195)--(-0.88356,-0.07285)--(-0.67937,0.02924)--(-0.5262,-0.04735)--(-0.21571,0.10788)--(0.0021,-0.00104)--(0.4113,0.20357)--(0.68478,0.06683)--(0.98842,0.21866)--(1.08513,0.1703)--(1.83023,1.66054) ..controls (1.97322,1.94412) and (2.04973,2.25977).. (2.04973,2.58333) ..controls (2.04973,3.12698) and (1.83374,3.6483).. (1.44934,4.03268) ..controls (1.06496,4.41707) and (0.54364,4.63306).. (0,4.63306) ..controls (-0.54364,4.63306) and (-1.06496,4.41707).. (-1.44934,4.03268) ..controls (-1.83374,3.6483) and (-2.04973,3.12698).. (-2.04973,2.58333) ..controls (-2.04973,2.26457) and (-1.97548,1.95349).. (-1.83656,1.67319)--(-0.99025,-0.0195)--cycle;
\draw  [<->] (0.10544,0)--(1,0);
\draw  [<->] (-1.33333,0.70181)--(-1.33333,1.66667);
\node [] at (0,0.18149) {$z$};
\fill  [line width=3.01126pt] (0,0) circle (1.50563pt);
\node [] at (-1.13664,1.16667) {$\rho$};
\node [] at (0.49245,-0.21567) {$\rho$};
\node [] at (2.74066,1) {slope $M$};
\node [] at (1.10034,-1.4) {slope $M_0$};
\node [] at (0,2.58333) {$T(z,\rho)$};
\node [] at (2.35042,0.16496) {$\partial V$};
\end{tikzpicture}
\caption{The tent $T(z,\rho)$.}\label{fig:tent}
\end{center}
\end{figure}


%
%
%
Then $T(z,\rho)$ is a bounded, simply connected domain with primary Lipschitz constant~$M_0$, and so the Dirichlet problem~\eqref{eqn:Dirichlet:tilde} or Neumann problem~\eqref{eqn:Neumann:tilde} in $T(z,\rho)$ is well posed. 

We now cover $\Xi$ by tents of this form.

\begin{lem}\label{lem:tent:cover}
Let $\Omega$ be a Lipschitz domain, let $x_0\in\Omega$ with $\dist(x_0,\partial\Omega)<\diam\partial\Omega$, and let $\Xi$ be as given above. If $h>0$ is small enough, depending only on the Lipschitz character of~$\Omega$ and not on our choice of $x_0\in\Omega$, then there exist constants $C_1>0$, $C_2>0$ and a (finite or countable) set of points $\{z_j\}\subset\partial\Omega$ such that the following conditions hold.
\begin{itemize}
\item $\Xi\subset\bigcup_{j} T(z_j,\abs{z_j-x_0}/2C_1)$.
\item if $x\in\Xi$ then $x\in T(z_j,\abs{z_j-x_0}/C_1)$ for at most $C_2$ values of~$j$.
\item If $k$ is an integer, then $2^k\leq \abs{z_j-x_0}\leq 2^{k+1}$ for at most $C_2$ values of~$j$.
\item $T(z_j,\abs{z_j-x_0}/C_1)\subset\Omega$ for all $j$, and so $\vec u$ is defined in $ T(z_j,\abs{z_j-x_0}/C_1)$.
\item If $T=T(z_j,\abs{z_j-x_0}/C_1)$ and $\widetilde T=\widetilde T(z_j,\abs{z_j-x_0}/C_1)$, then $\partial T \cap \partial V=\partial T\cap\partial\Omega=\partial T \setminus \partial \widetilde T$.
\item $T(z_j,\abs{z_j-x_0}/C_1)\cap B(x_0,\delta/2)=\emptyset$, and so $L\vec u=0$ in~$T(z_j,\abs{z_j-x_0}/C_1)$.
\end{itemize}
\end{lem}

\begin{proof}
We may ensure $T(z_j,\abs{z_j-x_0}/C_1)\cap B(x_0,\delta/2)=\emptyset$ by choosing $C_1$ large enough that $T(z,\rho/C_1)\subset B(z,\rho/2)$ for all $z\in\partial\Omega$ and all $\rho>0$. 

Let $x\in\Xi=\{x\in\Omega:\dist(x,\partial\Omega)\leq h\abs{x-x_0}\}$.
An elementary argument involving the triangle inequality shows that, if $z\in\partial\Omega$ with $\abs{x-z}=\dist(x,\partial\Omega)$, then $\abs{x-z}\leq (h/(1-h))\abs{z-x_0}$.  Thus, $x\in B(z,(h/(1-h))\abs{z-x_0})\cap\Omega$. For any given $C_1>0$, we may choose $h$ small enough that $B(z,(h/(1-h))\abs{z-x_0})\subset \widetilde T(z,\abs{z-x_0}/4C_1)$. Thus, 
\begin{equation*}\Xi\subset\bigcup_{z\in\partial\Omega} \widetilde T(z,\abs{z-x_0}/4C_1)\cap\Omega,\end{equation*}
and we may choose $z_j$ such that 
\begin{equation*}\Xi\subset\bigcup_{j=1}^\infty \widetilde T(z_j,\abs{z_j-x_0}/2C_1)\cap\Omega\end{equation*}
and such that if $x\in\R^\dmn$ then $x\in \widetilde T(z_j,\abs{z_j-x_0}/C_1)$ for at most $C_3$ values of~$j$, for some large constant~$C_3$; thus, $\abs{z_j-x_0}\approx 2^k$ for at most $C_2=C_3C_4$ values of~$j$, where $C_4\geq 1$.

It remains to choose $C_1$ large enough that 
\begin{equation*}\widetilde T(z_j,\abs{z_j-x_0}/C_1)\cap\Omega = \widetilde T(z_j,\abs{z_j-x_0}/C_1)\cap V =  T(z_j,\abs{z_j-x_0}/C_1)\end{equation*} for any such~$j$.

If $\Omega$ is a Lipschitz graph domain then $V=\Omega$ and there is nothing to do. Otherwise, observe that $\abs{z_j-x_0}<2\diam\partial\Omega$. We may choose $C_1$ large enough that $\widetilde T(z,2\diam\partial\Omega/C_1)\cap\Omega =  T(z,2\diam\partial\Omega/C_1)$ for all $z\in\partial\Omega$. This completes the proof.
\end{proof}

Thus, given the above choices of~$C_1$ and~$h$, if $\dist(x_0,\partial\Omega)<\diam\partial\Omega$ then the right-hand side of
\begin{multline}\label{eqn:atoms:53}
\int_\Xi \abs{\nabla^m \vec u(x)}^p\,\dist(x,\partial\Omega)^{p-1-p\smooth}\,dx
\\\leq C_2
\sum_{j} \int_{T(z_j,\abs{z_j-x_0}/2C_1)} \abs{\nabla^m \vec u(x)}^p\,\dist(x,\partial\Omega)^{p-1-p\smooth}\,dx\end{multline}
is meaningful and the inequality is valid. We need only bound the sum on the right hand side.

Fix some such~$j$. Let $R=\abs{z_j-x_0}/C_1$, and let $T_\mu= T(z_j,\mu R)$ and $\widetilde T_\mu=\widetilde T(z_j,\mu R)$ for any $1/2\leq \mu\leq 1$. Recall that $T_1\subset\Omega$ and that $\arr\Phi=0$ in~$T_1$. We wish to bound the integral over~$T_{1/2}$. 
By H\"older's inequality and because $p<\widetilde q_+$,
\begin{multline*}
\int_{T_{1/2}} \abs{\nabla^m \vec u(x)}^p\,\dist(x,\partial\Omega)^{p-1-p\smooth}\,dx
\\\leq
	\smash{\biggl(\int_{T_{1/2}} \abs{\nabla^m \vec u(x)}^{\widetilde q_+}\,\dist(x,\partial\Omega)^{\widetilde q_+-1-\widetilde q_+\widetilde \sigma_+}\,dx
	\biggr)^{p/\widetilde q_+}}\vphantom{|}
	\\\qquad\times
	\biggl(
	\int_{T_{1/2}} \dist(x,\partial\Omega)^{(\widetilde \sigma_+-\smooth)p\widetilde q_+/(\widetilde q_+-p)-1}\, dx 
	\biggr)^{1-p/\widetilde q_+}
.\end{multline*}
Evaluating the second integral, we have that
{\multlinegap=0pt\begin{multline*}
\int_{T_{1/2}} \abs{\nabla^m \vec u(x)}^p\,\dist(x,\partial\Omega)^{p-1-p\smooth}\,dx
\\\leq
	CR^{\pdmnMinusOne(1-p/\widetilde q_+)+(\widetilde \sigma_+-\smooth)p}\biggl(\int_{T_{1/2}} \abs{\nabla^m \vec u(x)}^{\widetilde q_+}\,\dist(x,\partial\Omega)^{\widetilde q_+-1-\widetilde q_+\widetilde \sigma_+}\,dx
	\biggr)^{p/\widetilde q_+}
.\end{multline*}}%
We wish to bound  the right-hand side.

\begin{lem}
If $x\in T_{1/2}$ and $3/4\leq \mu\leq 1$, then $\dist(x,\partial\Omega)\approx \dist(x,\partial T_\mu)$.
\end{lem}

\begin{proof}
Observe that  $T_1\subset\Omega$, and so if $x\in T_{1/2}$ and $1/2\leq\mu\leq 1$ then $\dist(x,\partial\Omega)\geq \dist(x,\partial T_\mu)$.

Conversely, suppose that $x\in T_{1/2}$ and $3/4\leq \mu\leq 1$. Then either $\dist(x,\partial T_\mu) = \dist(x,\partial T_\mu\cap\partial\Omega)\geq\dist(x,\partial\Omega)$ or $\dist(x,\partial T_\mu) = \dist(x,\partial T_\mu\cap\partial\widetilde T_\mu)$. 

If $\dist(x,\partial T_\mu) = \dist(x,\partial T_\mu\cap\partial\widetilde T_\mu)$, then $\dist (x,\partial T_\mu) \geq \dist(T_{1/2},\partial \widetilde T_{3/4})=R/C$. 
But $\dist(x,\partial\Omega)<RC$ for all $x\in T_{1/2}$,  and so $\dist(x,\partial\Omega) \leq C\dist(x,\partial T_\mu) $ for all $x\in T_{1/2}$ and all $3/4\leq \mu\leq 1$, as desired.
\end{proof}
Thus, if $3/4\leq\mu\leq 1$ then
\begin{multline*}
\biggl(\int_{T_{1/2}} \abs{\nabla^m \vec u(x)}^p\,\dist(x,\partial\Omega)^{p-1-p\smooth}\,dx\biggr)^{\widetilde q_+/p}
\\\leq
	CR^{\pdmnMinusOne(\widetilde q_+/p-1)+(\widetilde \sigma_+-\smooth)\widetilde q_+}
	\int_{T_\mu} \abs{\nabla^m \vec u(x)}^{\widetilde q_+}\,\dist(x,\partial T_\mu)^{\widetilde q_+-1-\widetilde q_+\widetilde \sigma_+}\,dx
.\end{multline*}

We now wish to apply the bounds \eqref{eqn:Dirichlet:tilde} and~\eqref{eqn:Neumann:tilde} to bound the right hand side. Thus, we must show that $\vec u\in \dot W^{\widetilde q_+,\widetilde\sigma_+}_{m,av}(T_\mu)$ for some $3/4\leq \mu\leq 1$. 
Because $\vec u\in \dot W^{q_-,\sigma_-}_{m,av}(\Omega)$, we have that $\vec u\in \dot W^2_{m,loc}(T_\mu)$ for any $0<\mu<1$. Furthermore,
$L\vec u=0$ in~$T_\mu$, and so by Lemma~\ref{lem:unaveraged},
\begin{equation*}\doublebar{\vec u}_{\dot W^{\widetilde q_+,\widetilde\sigma_+}_{m,av}(T_\mu)}^{\widetilde q_+}
\approx \int_{T_\mu} \abs{\nabla^m \vec u(x)}^{\widetilde q_+} \dist(x,\partial T_\mu)^{\widetilde q_+-1-\widetilde q_+\widetilde \sigma_+} \,dx.\end{equation*}
Observe that
\begin{multline*}
\fint_{3/4}^{5/6} \int_{T_\mu} \abs{\nabla^m \vec u(x)}^{\widetilde q_+} \dist(x,\partial T_\mu)^{\widetilde q_+-1-\widetilde q_+\widetilde \sigma_+} \,dx\,d\mu
\\\leq 
12\int_{T_{5/6}} \abs{\nabla^m \vec u(x)}^{\widetilde q_+} \int_{3/4}^{5/6} \dist(x,\partial T_\mu)^{\widetilde q_+-1-\widetilde q_+\widetilde \sigma_+}\,d\mu \,dx
\\\leq C \int_{T_{5/6}} \abs{\nabla^m \vec u(x)}^{\widetilde q_+} \dist(x,\partial \Omega)^{\widetilde q_+-1-\widetilde q_+\widetilde \sigma_+} \,dx
.\end{multline*}
As in Section~\ref{sec:far}, let $\vec u_+\in {\dot W^{q_+,\sigma_+}_{m,av}(\Omega)}$ be the solution to the problem \eqref{eqn:Dirichlet:q+} or~\eqref{eqn:Neumann:q-}. By \cite[Corollary~\ref*{P:cor:unique:extrapolate}]{Bar16pA} or by assumption, we have that $\vec u=\vec u_+$ and so $\vec u\in {\dot W^{q_+,\sigma_+}_{m,av}(\Omega)}$.
If $\widetilde\sigma_+=\sigma_+$ and $\widetilde q_+=q_+$, then by the bound~\eqref{eqn:unaveraged<averaged},
\begin{equation*}\int_{T_{5/6}} \abs{\nabla^m \vec u(x)}^{\widetilde q_+} \dist(x,\partial \Omega)^{\widetilde q_+-1-\widetilde q_+\widetilde \sigma_+} \,dx
\leq C\doublebar{\vec u}_{\dot W^{q_+,\sigma_+}_{m,av}(\Omega)}^{ q_+}<\infty\end{equation*}
and so for almost every $\mu\in (3/4,5/6)$ we have that
\begin{equation}\label{eqn:finite} 
\doublebar{\vec u}_{\dot W^{\widetilde q_+,\widetilde\sigma_+}_{m,av}(T_\mu)}^{\widetilde q_+}
\leq C\int_{T_\mu} \abs{\nabla^m \vec u(x)}^{\widetilde q_+} \dist(x,\partial T_\mu)^{\widetilde q_+-1-\widetilde q_+\widetilde \sigma_+} \,dx<\infty .\end{equation}
Thus, we may apply the bounds \eqref{eqn:Dirichlet:tilde} or~\eqref{eqn:Neumann:tilde} to~$\vec u$ in~$T_\mu$.

If $\widetilde \sigma_+<\sigma_+$, then by Lemma~\ref{lem:embedding:local} with $p=\widetilde q_+$, $\smooth=\widetilde\sigma_+$, and because $T_{5/6}\subset B(z_j,CR)$, we have that 
\begin{multline}\label{eqn:finite:2}\int_{T_{5/6}} \abs{\nabla^m \vec u(x)}^{\widetilde q_+} \dist(x,\partial \Omega)^{\widetilde q_+-1-\widetilde q_+\widetilde \sigma_+} \,dx
\\\leq CR^{\pdmnMinusOne(1-\widetilde q_+/q_+) +\widetilde q_+\sigma_+-\widetilde q_+\widetilde\sigma_+} \doublebar{\vec u}_{\dot W^{q_+,\sigma_+}_{m,av}(\Omega)}^{\widetilde q_+}<\infty\end{multline}
and we may still apply the bounds \eqref{eqn:Dirichlet:tilde} or~\eqref{eqn:Neumann:tilde}.

Thus, if $\Tr_{m-1}^\Omega\vec u=0$ and the inequality \eqref{eqn:Dirichlet:tilde} is valid, then 
\begin{multline}
\label{eqn:Dirichlet:tent:boundary}
\biggl(\int_{T_{1/2}} \abs{\nabla^m \vec u(x)}^p\,\dist(x,\partial\Omega)^{p-1-p\smooth}\,dx\biggr)^{{\widetilde q_+/p}}
\\\leq
	CR^{\pdmnMinusOne(\widetilde q_+/p-1)+(\widetilde \sigma_+-\smooth)\widetilde q_+}
	\int_{3/4}^{5/6}
	\doublebar{\Tr_{m-1}^{T_\mu}\vec u}_{\dot W\!A^{\widetilde q_+}_{m-1,\widetilde \sigma_+}(\partial T_\mu)}^{\widetilde q_+}\,d\mu
.\end{multline}
If $\M_{\mat A,\arr\Phi}^\Omega\vec u=0$ and the inequality \eqref{eqn:Neumann:tilde} is valid, then 
\begin{multline}
\label{eqn:Neumann:tent:boundary}
\biggl(\int_{T_{1/2}} \abs{\nabla^m \vec u(x)}^p\,\dist(x,\partial\Omega)^{p-1-p\smooth}\,dx\biggr)^{\widetilde q_+/p}
\\\leq
	CR^{\pdmnMinusOne(\widetilde q_+/p-1)+(\widetilde \sigma_+-\smooth)\widetilde q_+}\int_{3/4}^{5/6}
	\doublebar{\M_{\mat A,0}^{T_\mu} \vec u}_{\dot N\!A^{\widetilde q_+}_{m-1,\widetilde \sigma_+-1}(\partial T_\mu)}^{\widetilde q_+}\,d\mu
.\end{multline}
By the bounds~\eqref{eqn:finite} and~\eqref{eqn:finite:2} and \cite[Theorems~5.1 and~7.1]{Bar16pB}, $\Tr_{m-1}^{T_\mu}\vec u\in {\dot W\!A^{\widetilde q_+}_{m-1,\widetilde \sigma_+}(\partial T_\mu)}$ and  ${\M_{\mat A,0}^{T_\mu} \vec u}\in{\dot N\!A^{\widetilde q_+}_{m-1,\widetilde \sigma_+-1}(\partial T_\mu)} $ for almost every $\mu$ with $3/4\leq \mu\leq 5/6$.


To complete the bounds on $\nabla^m\vec u$ in $T_{1/2}$, we must bound the respective right hand sides.

\begin{lem}\label{lem:atoms:besov:Dirichlet}
Let $0<\sigma<1$ and let $1< q< \infty$, and let $T$ be a bounded simply connected Lipschitz domain.

If $\arr f\in \dot W\!A^{q}_{m-1,\sigma}(\partial T)$, and if $\nabla_\tau \arr f\in L^{q}(\partial T)$, where $\nabla_\tau $ denotes the gradient along~$\partial T$, then 
\begin{equation*}
\doublebar{\arr f}_{\dot W\!A^{q}_{m-1,\sigma}(\partial T)}
\leq C \,(\diam T)^{1-\sigma} \doublebar{\nabla_\tau f}_{L^{q}(\partial T)}
.\end{equation*}
\end{lem}

\begin{proof}
Recall from \cite[Section~2.2]{Bar16pB} that if $0<\sigma<1$ and $q\geq 1$, then 
\begin{equation*}\doublebar{\arr f}_{\dot W\!A^{q}_{m-1,\sigma}(\partial T)}^q
=\doublebar{\arr f}_{\dot B^{q,q}_{\sigma}(\partial T)}^q
\approx \int_{\partial T} \int_{\partial T} \frac{\abs{\arr f(x)-\arr f(y)}^q}{\abs{x-y}^{\dmnMinusOne+q\sigma}}\,d\sigma(x)\,d\sigma(y).
\end{equation*}
%
We may assume without loss of generality that $\int_{\partial T}\arr f\,d\sigma=0$.
By rescaling, it suffices to prove this theorem in the case where $\diam T=1$. 

Recall that $\dot W^q_1(\partial T)$ denotes a homogeneous Sobolev space.
Let $W^q_{1}(\partial T)$ denote the inhomogeneous Sobolev space with norm 
\begin{equation*}\doublebar{f}_{W^q_{1}(\partial T)} 
=\doublebar{f}_{L^q(\partial T)}+\doublebar{f}_{\dot W_1^q(\partial T)}=\doublebar{f}_{L^q(\partial T)} +\doublebar{\nabla_\tau f}_{L^q(\partial T)}.\end{equation*}
It is well known (see, for example, \cite[formulas~(2.401), (2.421), (2.490)]{MitM13A}) that
\begin{align*}
\doublebar{f}_{(L^q(\partial T),W^q_1(\partial T))_{\sigma,q}}
&\approx 
	\doublebar{f}_{L^q(\partial T)}+\doublebar{f}_{\dot B^{q,q}_\sigma(\partial T)}
\end{align*}
where $(\,\cdot\,,\,\cdot\,)_{\sigma,q}$ denotes the real interpolation functor of Lions and Peetre defined in, for example, \cite[Chapter~3]{BerL76}. By standard properties of interpolation spaces, 
\begin{equation*}
\doublebar{f}_{(L^q(\partial T),W^q_1(\partial T))_{\sigma,q}}
\leq C\doublebar{f}_{L^q(\partial T)}^{1-\sigma} \doublebar{f}_{W_1^q(\partial T)}^{\sigma}.
\end{equation*}
By the Poincar\'e inequality, $\doublebar{f}_{L^q(\partial T)}\leq C\doublebar{\nabla_\tau f}_{L^q(\partial T)}$. Thus,
$\doublebar{f}_{\dot B^{q,q}_\sigma(\partial T)}\leq C\doublebar{\nabla_\tau f}_{L^q(\partial T)}$,
as desired.
\end{proof}

\begin{lem}\label{lem:atoms:besov:Neumann}
Let $z\in\partial\Omega$ and $\rho>0$, and let $T=T(z,\rho)$. Recall the Lipschitz graph domain~$V$ in the definition of~$T$. Suppose that $\nabla^m\vec u\in L^1(T)$, that $L\vec u=0$ in~$T$, and that $\M_{\mat A,0}^{T} \vec u=0$ on $\partial T\cap\partial V$.

If $1<q\leq \infty$ and $0<\sigma<1$, then
\begin{equation*}
\doublebar{\M_{\mat A,0}^{T} \vec u}_{\dot N\!A^{q}_{m-1,\sigma-1}(\partial T)}^q\leq 
C\int_{T} \abs{\nabla^m \vec u(x)}^q \dist(x,\partial T\setminus\partial V)^{q-1-q\sigma}\,dx.
\end{equation*}
\end{lem}

\begin{proof}
By the duality characterization of $\dot N\!A^{q}_{m-1,\sigma-1}$ (see \cite[Section~2.2]{Bar16pB}), we have that
\begin{equation*}\doublebar{\M_{\mat A,0}^{T} \vec u}_{\dot N\!A^{q}_{m-1,\sigma-1}(\partial T)}
\approx \sup_{\arr\varphi\in {\dot W\!A^{q'}_{m-1,1-\sigma}(\partial T)}} \frac{1}{\doublebar{\arr \varphi}_{\dot W\!A^{q'}_{m-1,1-\sigma}(\partial T)}}\abs{\langle \arr\varphi, \M_{\mat A,0}^{T} \vec u\rangle_{\partial T}}
\end{equation*}
where $1/q+1/q'=1$.

Let ${\arr\varphi\in {\dot W\!A^{q'}_{m-1,1-\sigma}(\partial T)}}$.
By \cite[Theorem~4.1]{Bar16pB}, there are functions $\vec\Phi_i\in \dot W^{q',1-\sigma,\infty}_{m,av}(T)$ and $\vec\Phi_e\in \dot W^{q',1-\sigma,\infty}_{m,av}(\R^\dmn\setminus\overline T)$ such that $\Tr_{m-1}^T\vec\Phi_i = \Tr_{m-1}^{\R^\dmn\setminus\overline T}\vec\Phi_e = \arr \varphi$. Here $\dot W^{q',1-\sigma,r}_{m,av}(T)$ and $L^{q',1-\sigma,r}_{av}(T)$ are defined analogously to $\dot W^{q',1-\sigma}_{m,av}(T)$ and $L^{q',1-\sigma}_{av}(T)$, with the norm
\begin{gather*}
\doublebar{\arr \Psi}_{L^{q',1-\sigma,r}_{av}(T)}^{q'}
= \int_T \bigl(\doublebar{\arr \Psi}_{L^r(B(x,\dist(x,\partial T)/2))}\bigr)^{q'} \dist(x,\partial T)^{q'-1-q'(1-\sigma)-\pdmn q'/r}\,dx
.\end{gather*}

Recall the region $\widetilde T=\widetilde T(z,\rho)$ of formula~\eqref{eqn:T}. Observe that $\widetilde T$ is also a bounded, simply connected Lipschitz domain. Furthermore, if $x\in T$, then $\partial T\cap \partial\widetilde T = \partial T \setminus \partial V$ and so $\dist(x,\partial T\setminus \partial V) \approx \dist(x,\partial\widetilde T)$.
Let 
\begin{equation*}\vec\Phi(x) = \begin{cases}\vec\Phi_i(x),&x\in T,\\ \vec\Phi_e(x), &x\in\widetilde T\setminus\overline T.\end{cases}\end{equation*}

\begin{clm}\label{clm:1} We claim that $\vec\Phi \in \dot W^{q',1-\sigma,1}_{m,av}(\widetilde T)$, and that its norm depends only on $\doublebar{\vec \Phi_i}_{\dot W^{q',1-\sigma,\infty}_{m,av}( T)}$ and $\doublebar{\vec \Phi_e}_{\dot W^{q',1-\sigma,\infty}_{m,av}(\R^\dmn\setminus\overline T)}$.
\end{clm}

Suppose that this claim is true. Again by \cite[Theorem~4.1]{Bar16pB}, there is some $\vec F\in \dot W^{q',1-\sigma,\infty}_{m,av}(\widetilde T)$ with $\Tr_{m-1}^{\widetilde T}\vec F=\Tr_{m-1}^{\widetilde T}\vec \Phi$. Let 
\begin{equation*}\arr G(x) = \begin{cases}\mat A(x)\nabla^m \vec u(x),&x\in T,\\0&\text{otherwise}.\end{cases}\end{equation*}
Then 
\begin{equation*}{\langle \arr\varphi,\M_{\mat A,0}^{T} \vec u\rangle_{\partial T}}
= \langle \nabla^m \vec\Phi, \mat A\nabla^m \vec u\rangle_T
=\langle \nabla^m \vec \Phi,\arr G\rangle_{\widetilde T}
=\langle \nabla^m \vec F,\arr G\rangle_{\widetilde T}
\end{equation*}
and so
\begin{align*}
\abs{\langle \arr\varphi, \M_{\mat A,0}^{T} \vec u\rangle_{\partial T}}
&\leq 
	C \doublebar{\nabla^m\vec F}_{L^{q',1-\sigma,\infty}_{av}(\widetilde T)}
	\doublebar{\arr G}_{L^{q,\sigma,1}_{av}(\widetilde T)}
\\&\leq 
	C\doublebar{\arr \varphi}_{\dot W\!A^{q'}_{m-1,1-\sigma}(\partial T)}
	\doublebar{\1_T\nabla^m\vec u}_{L^{q,\sigma,1}_{av}(\widetilde T)}.
\end{align*}
Because $q\geq 1$, arguing analogously to the proof of the bound~\eqref{eqn:unaveraged<averaged}, we may show that 
\begin{equation*}\doublebar{\1_T\nabla^m\vec u}_{L^{q,\sigma,1}_{av}(\widetilde T)}^q
\leq C\int_T \abs{\nabla^m \vec u(x)}^q \dist(x,\partial T\setminus\partial V)^{q-1-q\sigma}\,dx\end{equation*}
and so the proof is complete.

We now must establish the claim; that is, we must bound
\begin{equation*}\int_{\widetilde T} \biggl(\fint_{B(x,\widetilde T)} \abs{\nabla^m \vec\Phi}\biggr)^{q'} \dist(x,\partial\widetilde T)^{q'-1-q'(1-\sigma)}\,dx\end{equation*}
where $B(x,\widetilde T) = B(x,a\dist(x,\partial\widetilde T))$ for some $0<a<1$.

Let $0<\eta<1$.
Let 
\begin{align*}
T_1 &= \{x\in \widetilde T: \eta\dist(x,\partial \widetilde T)
\leq \dist(x,\partial T)\leq\frac{1}{\eta}\dist(x,\partial \widetilde T)\}
,\\
T_2 &= \{x\in \widetilde T: \dist(x,\partial T)< \eta\dist(x,\partial \widetilde T)\}
,\\
T_3 &= \{x\in \widetilde T: \dist(x,\partial \widetilde T)< \eta\dist(x,\partial T)\}
.\end{align*}

If $x\in T_1$, then $\dist(x,\partial T)\approx\dist(x,\partial\widetilde T)$. Thus,
\begin{align*}
\int_{T_1\cap T} \biggl(\fint_{B(x,\widetilde T)} \abs{\nabla^m \vec\Phi}\biggr)^{q'} \dist(x,\partial\widetilde T)^{q'-1-q'(1-\sigma)}\,dx
&\leq C\doublebar{\nabla^m\vec \Phi_i}_{L^{q',1-\sigma,\infty}_{av}(T)}
,\\
\int_{T_1\setminus T} \biggl(\fint_{B(x,\widetilde T)} \abs{\nabla^m \vec\Phi}\biggr)^{q'} \dist(x,\partial\widetilde T)^{q'-1-q'(1-\sigma)}\,dx
&\leq C\doublebar{\nabla^m\vec \Phi_e}_{L^{q',1-\sigma,\infty}_{av}(\R^\dmn\setminus\overline T)}
.\end{align*}

We now consider $T_2$.
Let $\Delta=\{x'\in\R^\dmnMinusOne:(x',\psi(x'))\in \partial T\cap \partial V\}$. We may choose $\eta$ small enough that, if $(x',t)\in T_2$, then $x'\in\Delta^O$, where $\Delta^O$ is the interior (in $\R^\dmnMinusOne$) of~$\Delta$. 

Let $\mathcal{G}$ be a decomposition of $\Delta^O$ into Whitney cubes, so $\Delta^O=\cup_{Q\in\mathcal{G}} Q$, $\ell(Q)\approx \dist(Q,\R^\dmnMinusOne\setminus\Delta^O)$, and if $Q$ and $R$ are distinct cubes in~$\mathcal{G}$ then $Q$ and $R$ have disjoint interiors. 

Let $\tau(Q)=\{(x',t):x'\in Q, \psi(x')-c_0\ell(Q)<t<\psi(x')+c_0\ell(Q)\}$ be two-sided versions of the tents in \cite[Lemma~3.7]{Bar16pB}. 
By the bound \cite[formula~(3.8)]{Bar16pB}, 
we have that if $\arr\Psi\big\vert_V\in {L^{q',\smooth,1}_{av}(V)}$ and $\arr\Psi\big\vert_{\R^\dmn\setminus \overline V}\in {L^{q',\smooth,1}_{av}(\R^\dmn\setminus \overline V)}$, then
\begin{equation*}\sum_{Q\in \mathcal{G}} \biggl(\fint_{\tau(Q)} \abs{\arr\Psi}\biggr)^{q'} \ell(Q)^{\dmnMinusOne+q'-q'\smooth} \leq C \doublebar{\arr\Psi}_{L^{q',\smooth,1}_{av}(V)}+\doublebar{\arr\Psi}_{L^{q',\smooth,1}_{av}(\R^\dmn\setminus \overline V)}.\end{equation*}
If $a$ and $\eta$ are small enough, then
\begin{multline*}\int_{T_2} \biggl(\fint_{B(x,\widetilde T)} \abs{\nabla^m \vec\Phi}\biggr)^{q'} \dist(x,\partial\widetilde T)^{q'-1-q'(1-\sigma)}\,dx
\\\leq
C\sum_{Q\in\mathcal{G}} \smash{\biggl(\fint_{\tau(Q)} \abs{\nabla^m\vec\Phi}\biggr)^{q'} \ell(Q)^{\dmnMinusOne+q'-q'(1-\sigma)}}
\end{multline*}
which by the above remarks is at most 
\begin{equation*}C\doublebar{\nabla^m\vec \Phi_i}_{L^{q',1-\sigma,\infty}_{av}(T)}+C\doublebar{\nabla^m\vec \Phi_e}_{L^{q',1-\sigma,\infty}_{av}(\R^\dmn\setminus\overline T)}.\end{equation*}

Finally, we come to $T_3$. Observe that $T_3\subset \widetilde T\setminus \overline T$ is a region lying near~$\partial\widetilde T$. We may write $T_3\subset \cup_{R\in \mathcal{H}} R$, where $\mathcal{H}$ is a collection of pairwise-disjoint cubes in $\R^\dmn$ (not $\R^\dmnMinusOne$) that satisfy $\ell(R)\approx \dist(R,\partial T)$.

Then 
\begin{multline*}\int_{T_3} \biggl(\fint_{B(x,\widetilde T)} \abs{\nabla^m \vec\Phi}\biggr)^{q'} \dist(x,\partial\widetilde T)^{q'-1-q'(1-\sigma)}\,dx
\\\leq
C\sum_{R\in\mathcal{H}} \int_R \biggl(\fint_{B(x,\widetilde T)} \abs{\nabla^m \vec\Phi_e}^2\biggr)^{q'/2} \dist(x,\partial\widetilde T)^{q'-1-q'(1-\sigma)}\,dx
\end{multline*}
and by \cite[Lemma~\ref*{P:lem:L:L-infinity}]{Bar16pB},
\begin{equation*}\int_R \biggl(\fint_{B(x,\widetilde T)} \abs{\nabla^m \vec\Phi}^2\biggr)^{q'/2} \dist(x,\partial\widetilde T)^{q'-1-q'(1-\sigma)}\,dx
\leq C\doublebar{\nabla^m\vec\Phi_e}_{L^\infty(R)}^{q'} \ell(R)^{\dmnMinusOne+q'\sigma}.\end{equation*}
But 
\begin{equation*}\sum_{R\in\mathcal{H}} \doublebar{\nabla^m\vec\Phi_e}_{L^\infty(R)}^{q'}  \ell(R)^{\dmnMinusOne+q'\sigma}
\leq C \doublebar{\nabla^m\vec\Phi_e}_ {L^{q',1-\sigma,\infty}_{av}(\R^\dmn\setminus\overline T)}.\end{equation*}
This completes the proof.
\end{proof}

By Lemma~\ref{lem:atoms:besov:Dirichlet} or Lemma~\ref{lem:atoms:besov:Neumann}, we may bound the integrands on the right hand sides of formulas~\eqref{eqn:Dirichlet:tent:boundary} and~\eqref{eqn:Neumann:tent:boundary}; evaluating the integrals yields the bound
\begin{equation*}%
\int_{T_{1/2}} \abs{\nabla^m \vec u(x)}^p\,\dist(x,\partial\Omega)^{p-1-p\smooth}\,dx
\\\leq
C R^{\dmn +p-1- p\smooth-\pdmn p/\widetilde q_+}
\doublebar{\nabla^m \vec u}_{L^{\widetilde q_+}(T_{5/6})}^p
.\end{equation*}
Because $\widetilde q_+\leq 2$, we may use H\"older's inequality to see that
\begin{equation*}%
\int_{T_{1/2}} \abs{\nabla^m \vec u(x)}^p\,\dist(x,\partial\Omega)^{p-1-p\smooth}\,dx
\\\leq
C R^{\dmn +p-1- p\smooth-\pdmn p/2}
\doublebar{\nabla^m \vec u}_{L^2(T_{5/6})}^p
.
\end{equation*}

We now apply {\cite[Lemmas~9 and~16]{Bar16}}; these lemmas are boundary Caccioppoli inequalities, with Dirichlet or Neumann boundary conditions. 
This yields the bound
\begin{equation*}
\int_{T_{1/2}} \abs{\nabla^m \vec u(x)}^p\,\dist(x,\partial\Omega)^{p-1-p\smooth}\,dx
\leq
	CR^{\dmnMinusOne-p\smooth-\pdmn p/2} \doublebar{\nabla^{m-1}(\vec u-\vec P)}_{L^2(T_{6/7})}^p
\end{equation*}
where $\vec P=0$ if $\Tr_{m-1}^\Omega\vec u=0$, and where $\vec P$ is an arbitrary polynomial of degree $m-1$ (and so $\nabla^{m-1}\vec P$ is an arbitrary constant array) if $\M_{\mat A,\arr\Phi}^\Omega\vec u=0$.

By Theorem~\ref{thm:Meyers}, we have that
\begin{multline}\label{eqn:atoms:3}
\int_{T_{1/2}} \abs{\nabla^m \vec u(x)}^p\,\dist(x,\partial\Omega)^{p-1-p\smooth}\,dx
\\\leq
	CR^{\dmnMinusOne-\pdmn p/q_- -\smooth p} \doublebar{\nabla^{m-1}(\vec u-\vec P)}_{L^{q_-}(T_{7/8})}^p
.\end{multline}

We must now bound $\nabla^{m-1} (\vec u-\vec P)$ in $T_{7/8}$. We will use the following lemma. 
\begin{lem}\label{lem:poincare:tents}
Let $\Omega$ be a Lipschitz domain and let $T=T(z,\rho)$ for some $z\in\partial\Omega$ and some $\rho>0$ small enough that $T(z,(8/7)\rho)\subset\Omega$ and $\partial T\cap \partial V=\partial T\cap\partial\Omega$. Let $v$ be a function defined in~$T$ and let $0\leq\tau\leq\rho(M_0-M)/M_0$.

Suppose that $1\leq q<\infty$ and $\sigma>0$.
Then
\begin{multline*}
\biggl(\int_T \abs{v(x)}^q\,dx\biggr)^{1/q}
\leq
C\rho^{1/q+\sigma}
\biggl(\int_T \abs{\nabla v(x)}^q \dist(x,\partial\Omega)^{q-1-q\sigma}\,dx\biggr)^{1/q}
\\
+ C\rho^{1/q}\biggl(\int_{(\partial V+(0,\tau))\cap  T} \abs{v(x)}^q\,d\sigma(x)\biggr)^{1/q}
.\end{multline*}
\end{lem}
Here $\partial V+(0,\tau)=\{x+(0,\tau):x\in\partial V\}$.
\begin{proof}[Proof of Lemma~\ref{lem:poincare:tents}]
Let $x_T$ be the vertex of the cone $\widetilde T(z,\rho)$, and let $S=\{\omega\in\R^\dmn:\abs\omega=1,\allowbreak\,x_T+s\omega\in T$ for some $s>0\}$. 
Then there is some $a(\omega)$ and $b(\omega)$ so that $T=\{x_T+s\omega:\omega \in S,\>a(\omega)<s<b(\omega)\}$.
Thus,
\begin{align*}
\int_T\abs{v(x)}^q\,dx
&=
	\int_S \int_{a(\omega)}^{b(\omega)} \abs{v(x_T+r\omega)}^{q} \,dr\,r^{\dmnMinusOne}  \,d\sigma(\omega)
\end{align*}
where $d\sigma(\omega)$ denotes surface measure on the unit sphere in~$\R^\dmn$. 

If $\omega\in S$ and $\tau\geq 0$ then there is a unique $\mu(\omega)>0$ such that $x_T+\mu(\omega)\,\omega\in \partial V+(0,\tau)$. If $0\leq \tau\leq \rho(M_0-M)/M_0$, then $x_T+\mu(\omega)\,\omega\in T$ and so $a(\omega)\leq\mu(\omega)\leq b(\omega)$; furthermore, 
${(\partial V+(0,\tau))\cap  T}=\{x_T+\mu(\omega)\,\omega:\omega \in S\}$.
Then 
\begin{multline*}
\biggl(\int_T\abs{v(x)}^q\,dx\biggr)^{1/q}
\leq
	\biggl(\int_S \int_{a(\omega)}^{b(\omega)} \abs[bigg]{\int_{\mu(\omega)}^r\abs{\nabla v(x_T+s\omega)}\,ds}^{q} \,r^{\dmnMinusOne} \,dr\,d\sigma(\omega)\biggr)^{1/q}
	\\+
	\biggl(\int_S \abs{v(x_T+\mu(\omega)\omega)}^{q} \int_{a(\omega)}^{b(\omega)} r^{\dmnMinusOne} \,dr\,d\sigma(\omega)\biggr)^{1/q}
.\end{multline*}
The second term on the right hand side is at most
\begin{equation*}C\biggl(\rho\int_{(\partial V+(0,\tau))\cap  T} \abs{v(x)}^q\,d\sigma(x)\biggr)^{1/q}\end{equation*}
where $d\sigma(x)$ denotes surface measure on $\partial V+(0,\tau)$.

Let $I$ denote the first term. If $q>1$, then by H\"older's inequality,
\begin{align*}I^q
&\leq
	C\int_S \int_{a(\omega)}^{b(\omega)} 
	{\int_{a(\omega)}^{b(\omega)}\abs{\nabla v(x_T+s\omega)}^{q} (s-a(\omega))^{q-1-q\sigma}\,ds}
	\\&\qquad\qquad\qquad\times
	\biggl(\int_{\mu(\omega)}^r (s-a(\omega))^{-1+q'\sigma}\,ds\biggr)^{q/q'}
	\,r^{\dmnMinusOne}\,dr \,d\sigma(\omega)
.\end{align*}
If $\sigma>0$ and $q'>0$, then the second integral $ds$ converges. Evaluating, we see that
\begin{align*}I^q
&\leq
	C\rho^{\dmn+q\sigma}
	\int_S 
	{\int_{a(\omega)}^{b(\omega)}\abs{\nabla v(x_T+s\omega)}^{q} (s-a(\omega))^{q-1-q\sigma}\,ds}\,d\sigma(\omega)
.\end{align*}
If $q=1$ and $\sigma>0$, then $q-1-q\sigma=-\sigma<0$ and it is straightforward to show that this inequality is still valid.

Observe that if $a(\omega)<s<b(\omega)$, then $s\approx\rho$ and $s-a(\omega)\approx \dist(x_T+s\omega,\partial\Omega)$.
Thus,
\begin{align*}I^q
&\leq
	C\rho^{1+q\sigma}
	\int_T\abs{\nabla v(x)}^{q} \dist(x,\partial\Omega)^{q-1-q\sigma}\,dx
.\end{align*}
This completes the proof.
\end{proof}

If $\Tr_{m-1}^\Omega \vec u=0$, then applying the bound~\eqref{eqn:atoms:3} and Lemma~\ref{lem:poincare:tents} with $v=\nabla^{m-1}\vec u$ and $\tau=0$ yields that
{\multlinegap=0pt\begin{multline}\label{eqn:atoms:4}
\int_{T_{1/2}} \abs{\nabla^m \vec u(x)}^p\,\dist(x,\partial\Omega)^{p-1-p\smooth}\,dx
\\\leq
	CR^{\pdmnMinusOne(1- p/q_-) -\smooth p+p\sigma_-}
	\biggl(\int_{T_{7/8}} \abs{\nabla^{m}\vec u(x)}^{q_-}\dist(x,\partial\Omega)^{q_--1-q_-\sigma_-}\biggr)^{p/q_-}
.\end{multline}}

If $\M_{\mat A}^\Omega \vec u=0$, then averaging over a range of~$\tau$ yields that 
{\multlinegap=0pt\begin{multline*}
\int_{T_{1/2}} \abs{\nabla^m \vec u(x)}^p\,\dist(x,\partial\Omega)^{p-1-p\smooth}\,dx
\\\leq
	CR^{\pdmnMinusOne(1- p/q_-) -\smooth p+p\sigma_-} \biggl(
\int_{T_{7/8}} \abs{\nabla^m \vec u(x)}^{q_-} \dist(x,\partial\Omega)^{q_--1-q_-\sigma_-}\,dx\biggr)^{p/q_-}
\\+ CR^{\dmnMinusOne-\pdmn p/q_- -\smooth p} \biggl(\int_{U} \abs{\nabla^{m-1} \vec u(x)-\nabla^{m-1}\vec P(x)}^q_-\,dx
	\biggr)^{p/q_-}
\end{multline*}}%
for an open set~$U\subset T_{7/8}$ with $\dist(U,\partial\Omega)\geq R/C$. Choosing $\vec P$ appropriately and applying the Poincar\'e inequality, we see that the bound~\eqref{eqn:atoms:4} is still valid.

Because $q_-\leq 2$, by the bound~\eqref{eqn:unaveraged<averaged} we have that 
\begin{equation*}\int_{T_{7/8}} \abs{\nabla^m u(x)}^{q_-} \dist(x,\partial\Omega)^{q_--1-q_-\sigma_-}\,dx
\leq C \doublebar{\vec u}_{\dot W^{q_-,\smooth_-}_{m,av}(\Omega)}^{q_-}.\end{equation*}
By the bounds~\eqref{eqn:atoms:53} and~\eqref{eqn:atoms:4}, we have that
\begin{multline*}
\int_\Xi \abs{\nabla^m \vec u(x)}^p\,\dist(x,\partial\Omega)^{p-1-p\smooth}\,dx
\\\leq C
	\doublebar{\nabla^{m}\vec u}_{\dot W^{q_-,\smooth_-}_{m,av}(\Omega)}^p
\sum_{j} 
\abs{z_j-x_0}^{\dmnMinusOne -\smooth p-\pdmnMinusOne p/q_-+p\sigma_-} 
.\end{multline*}
Observe that  ${\dmnMinusOne -\smooth p-\pdmnMinusOne p/q_-+p\sigma_-}<0$.
Recall that for any $k\geq 0$, there are at most $C$ points $z_j$ with $2^k\leq \abs{z_j-x_0}\leq 2^{k+1}$. Thus, the sum may be bounded by a convergent geometric series, and we have that 
\begin{multline*}
\int_\Xi \abs{\nabla^m \vec u(x)}^p\,\dist(x,\partial\Omega)^{p-1-p\smooth}\,dx
\\\leq C
	\doublebar{\nabla^{m}\vec u}_{\dot W^{q_-,\sigma_-}_{m,av}(\Omega)}^p
\dist(x_0,\partial\Omega)^{\dmnMinusOne -\smooth p-\pdmnMinusOne p/q_-+p\sigma_-} 
.\end{multline*}
Recall that $\doublebar{\nabla^{m}\vec u}_{\dot W^{q_-,\smooth_-}_{m,av}(\Omega)}\leq C \doublebar{\arr\Phi}_{L^{q_-,\smooth_-}_{av}(\Omega)}$. Thus, by the bound~\eqref{eqn:whitney:q:sigma}, 
\begin{equation*}
\int_\Xi \abs{\nabla^m \vec u(x)}^p\,\dist(x,\partial\Omega)^{p-1-p\smooth}\,dx
\leq C
	\doublebar{\arr\Phi}_{L^{p,\smooth}_{av}(\Omega)}^p
.\end{equation*}
This completes the proof of Lemma~\ref{lem:atoms:boundary}.

\section{Proof of Theorem~\ref{thm:atoms}}
\label{sec:proof}

Let $\mathcal{G}$ be a grid of Whitney cubes in~$\Omega$; then $\Omega=\cup_{Q\in\mathcal{G}} Q$, the cubes in~$\mathcal{G}$ have pairwise-disjoint interiors, and if $Q\in\mathcal{G}$ then the side-length $\ell(Q)$ of~$Q$ satisfies $\ell(Q)\approx\dist(Q,\partial\Omega)$. As observed in \cite[Section~3]{Bar16pB},
if $0<p<\infty$ and $\arr H\in L^{p,\smooth}_{av}(\Omega)$, then
\begin{equation*}
\doublebar{\arr H}_{L^{p,\smooth}_{av}(\Omega)}
\approx \biggl( \sum_{Q\in\mathcal{G}} \biggl(\fint_Q \abs{\arr H}^2\biggr)^{p/2} \ell(Q)^{\dmnMinusOne+p-p\smooth}\biggr)^{1/p}
\end{equation*}
where the comparability constants depend on~$p$, $\smooth$, and the comparability constants for Whitney cubes in the relation $\ell(Q)\approx\dist(Q,\partial\Omega)$. 

Choose some $\arr H\in L^{p,\smooth}_{av}(\Omega)$.
For each $Q\in \mathcal{G}$, let $\vec u_Q \in {\dot W^{q_-,\sigma_-}_{m,av}(\Omega)}$ be as in Lemma~\ref{lem:atoms:boundary} with $\arr\Phi=\1_Q\arr H$. 

Because $p\leq 1$, we have that
\begin{equation*}\doublebar[Big]{\sum_{Q\in\mathcal{G}} \vec u_Q}_{\dot W^{p,\smooth}_{m,av}(\Omega)}^p 
\leq 
\sum_{Q\in\mathcal{G}} \doublebar{\vec u_Q}_{\dot W^{p,\smooth}_{m,av}(\Omega)}^p
\leq 
\sum_{Q\in\mathcal{G}} 
\frac{\doublebar{\vec u_Q}_{\dot W^{p,\smooth}_{m,av}}^p} {\doublebar{\1_Q\arr H}_{L^{p,\smooth}_{av}}^p} 
{\doublebar{\1_Q\arr H}_{L^{p,\smooth}_{av}(\Omega)}^p}
.\end{equation*}
By Lemmas~\ref{lem:unaveraged} and~\ref{lem:atoms:boundary}, we have that
\begin{equation*}
\doublebar{\vec u_Q}_{\dot W^{p,\smooth}_{m,av}(\Omega)}
\leq C
	\doublebar{\1_Q\arr H}_{L^{p,\smooth}_{av}(\Omega)}
\end{equation*}
for all~$Q$, and so
\begin{align*}
\doublebar[Big]{\sum_{Q\in\mathcal{G}} \vec u_Q}_{\dot W^{p,\smooth}_{m,av}(\Omega)}^p 
&\leq 
C \sum_{Q\in\mathcal{G}} {\doublebar{\1_Q\arr H}_{L^{p,\smooth}_{av}(\Omega)}^p}
\leq
C^2  {\doublebar{\arr H}_{L^{p,\smooth}_{av}(\Omega)}^p}
.\end{align*}
Thus, for any $\arr H\in L^{p,\smooth}_{av}(\Omega)$ there exists a solution $\vec u=\sum_Q\vec u_Q$ to the problem \eqref{eqn:Dirichlet:p:smooth} or~\eqref{eqn:Neumann:p:smooth} with boundary data $\arr f=0$ or $\arr g=0$. By \cite[Lemma~\ref*{P:lem:zero:boundary}]{Bar16pA}, we may extend to arbitrary boundary values. Finally, recall that by Lemma~\ref{lem:unique:atoms}, we have uniqueness of solutions to the problems \eqref{eqn:Dirichlet:p:smooth} or~\eqref{eqn:Neumann:p:smooth}. This completes the proof of Theorem~\ref{thm:atoms}.

\section{Known results in the notation of the present paper}
\label{sec:A0}

In Section~\ref{sec:A0:introduction}, we described new well posedness results arising from Theorem~\ref{thm:atoms} and from known results from \cite{MazMS10} and from \cite{MitMW11,MitM13B}. However, the results of \cite{MazMS10} and \cite{MitMW11,MitM13B} were stated not in terms of the spaces $L^{p,\smooth}_{av}(\Omega)$ and $\dot W^{p,\smooth}_{m,av}(\Omega)$ of the present paper, but in terms of other, related spaces. In \cite[Section~\ref*{P:sec:biharmonic}]{Bar16pA}, results in terms of  $L^{p,\smooth}_{av}(\Omega)$ and $\dot W^{p,\smooth}_{m,av}(\Omega)$ were derived from the results of~\cite{MitM13B}; in this section, we shall similarly derive results in terms of  $L^{p,\smooth}_{av}(\Omega)$ and $\dot W^{p,\smooth}_{m,av}(\Omega)$ from the results of~\cite{MazMS10}.

We begin by recalling the following result from \cite{MazMS10}.

\begin{thm}[{\cite[Theorem~8.1]{MazMS10}}]\label{thm:MazMS10}
Let $\Omega\subset\R^\dmn$ be a bounded Lipschitz domain and let $L$ be an elliptic differential operator of order~$2m$ of the form~\eqref{eqn:divergence}, defined in the weak sense of formula~\eqref{eqn:weak}, associated to coefficents~$\mat A$ that satisfy the ellipticity conditions \eqref{eqn:elliptic:bounded} and~\eqref{eqn:elliptic:everywhere}.


Then there is some $c>0$ such that, if $0<\smooth<1$, $1<p<\infty$, and 
\begin{equation}\label{eqn:MazMS10:bound}\delta(\mat A,\Omega) \leq c \frac{\smooth^2(1-\smooth)^2(1/p)(1-1/p)}{\smooth(1-\smooth)+(1/p)(1-1/p)}
\end{equation}
where $\delta(\mat A,\Omega)$ is as in formula~\eqref{eqn:MazMS10},
then the Dirichlet problem
\begin{equation*}
\left\{\begin{gathered}
L\vec u=F \text{ in }\Omega, \quad\partial_\nu^k\vec u=g_k \text { on $\partial\Omega$ for $0\leq k\leq m-1$},  
\\
\doublebar{\vec u}_{W^{p,\smooth}_m(\Omega)} \leq C\doublebar{\vec g}_{W^p_{m-1+\smooth}(\partial\Omega)} + C\doublebar{F}_{V^{p,\smooth}_{-m}(\Omega)}
\end{gathered}\right.
\end{equation*}
is well posed.
\end{thm}
Here ${V^{p,\smooth}_{-m}(\Omega)}$ is the dual space ${V^{p',1-\smooth}_{m}(\Omega)}^*$ to ${V^{p',1-\smooth}_{m}(\Omega)}$, and
{\allowdisplaybreaks
\begin{align*}
\doublebar{\vec u}_{W^{p,\smooth}_m(\Omega)} &= \biggl(\sum_{\abs\alpha\leq m} \int_{\Omega} \abs{\partial^\alpha \vec u(x)}^p\dist(x,\partial\Omega)^{p-1-p\smooth}\,dx\biggr)^{1/p},
\\
\doublebar{\vec u}_{V^{p,\smooth}_{m}(\Omega)} &= \biggl(\sum_{\abs\alpha\leq m} \int_{\Omega} \abs{\partial^\alpha \vec u(x)}^p\dist(x,\partial\Omega)^{p-1-p\smooth+p\abs\alpha-pm}\,dx\biggr)^{1/p}
.\end{align*}}%
For convenience, we will treat the case $\vec g=0$, avoiding the ${W^p_{m-1+\smooth}(\partial\Omega)}$ norms, and use \cite[Lemma~\ref*{P:lem:zero:boundary}]{Bar16pA} to contend with boundary values. 
By \cite[Theorem~7.8]{MazMS10},
\begin{equation}\label{eqn:V:W}
{V^{p,\smooth}_{m}(\Omega)}=\{\vec u\in {W^{p,\smooth}_m(\Omega)}: \Tr_k^\Omega\vec u=0 \text{ for all }0\leq k\leq m-1\}
.\end{equation}

We remark that, if $0<\smooth<1$ and $0<1/p<1$, then 
\begin{equation*}
\frac{1}{8}\min(s,1-s,1/p,1-1/p)^2\leq 
\frac{\smooth^2(1-\smooth)^2(1/p)(1-1/p)} {\smooth(1-\smooth)+(1/p)(1-1/p)} 
\end{equation*}
and that a bound on $\min(s,1-s,1/p,1-1/p)$ (rather than the more complicated condition \eqref{eqn:MazMS10:bound}) is more convenient to apply in the context of Theorem~\ref{thm:atoms}.

In this section we will derive the following well posedness result.
\begin{thm}\label{thm:MazMS10:translated} Let $\Omega$ and $\mat A$ be as in Theorem~\ref{thm:MazMS10}. Suppose furthermore that $\partial\Omega$ is connected.
If $1<p<\infty$, $0<\smooth<1$ and the condition~\eqref{eqn:MazMS10:bound} is satisfied, then the Dirichlet problem
\eqref{eqn:Dirichlet:p:smooth}
is well posed.
\end{thm}

\begin{proof}
By \cite[Theorems~\ref*{P:thm:exist:unique} and~\ref*{P:thm:unique:exist}]{Bar16pA}, we need only consider the case $p\leq 2$. By \cite[Lemma~\ref*{P:lem:zero:boundary}]{Bar16pA}, we need only consider the case $\arr f=0$.

We begin with uniqueness. Suppose that $L\vec u=0$ in $\Omega$, that $\Tr_{m-1}^\Omega\vec u=0$, and that $\vec u\in \dot W^{p,\smooth}_{m,av}(\Omega)$. We may normalize $\vec u$ so that $\Tr_k^\Omega\vec u=0$ for all $0\leq k\leq m-1$. It suffices to show that $\vec u\in V^{p,\smooth}_m(\Omega)$, for then $\vec u\in W^{p,\smooth}_m(\Omega)$ and so by Theorem~\ref{thm:MazMS10} must be zero.

Let $V=\{(x',t):t>\psi(x')\}$ be a Lipschitz graph domain and suppose that $\Tr_k^V \vec v=0$ for any $0\leq v\leq m-1$. We may write
\begin{multline*}\int_V \abs{\nabla^k \vec v(x)}^p \dist(x,\partial V)^{p-1-p\smooth+pk-pm}\,dx
\\\approx
\int_{\R^\dmnMinusOne} \int_0^\infty \abs{\nabla^k \vec v(x',t+\psi(x'))}^p t^{p-1-p\smooth+pk-pm}\,dt\,dx'.
\end{multline*}
If $0\leq k\leq m-1$, then
\begin{multline*}\int_V \abs{\nabla^k \vec v(x)}^p \dist(x,\partial V)^{p-1-p\smooth+pk-pm}\,dx
\\\leq C
\int_{\R^\dmnMinusOne} \int_0^\infty \abs[bigg]{\int_0^t\abs{\nabla^{k+1} \vec v(x',r+\psi(x'))}\,dr}^p t^{p-1-p\smooth+pk-pm}\,dt\,dx'.
\end{multline*}
By H\"older's inequality, for any $\theta\in\R$ with $p'\theta<1$, we have that
\begin{multline*}\int_V \abs{\nabla^k \vec v(x)}^p \dist(x,\partial V)^{p-1-p\smooth+pk-pm}\,dx'
\\\leq C
\int_{\R^\dmnMinusOne} \int_0^\infty 
\int_0^t\abs{\nabla^{k+1}\vec v(x',s+\psi(x'))}^p  r^{p\theta}\,dr \,
t^{p-1-p\smooth+pk-pm+p/p'-p\theta}\,dt\,dx'.
\end{multline*}
Changing the order of integration, we see that if $p-p\smooth+pk-pm+p/p'<p\theta$, then
\begin{multline*}\int_V \abs{\nabla^k \vec v(x)}^p \dist(x,\partial V)^{p-1-p\smooth+pk-pm}\,dx'
\\\leq C
\int_{\R^\dmnMinusOne} \int_0^\infty \abs{\nabla^{k+1}\vec v(x',r+\psi(x'))}^p 
r^{p-p\smooth+pk-pm+p/p'}\,dr \,dx'
\\\approx
\int_V \abs{\nabla^{k+1}\vec v(x)}^p \dist(x,\partial V)^{p-1-p\smooth+p(k+1)-pm}\,dx.
\end{multline*}
If $s>0$ and $k\leq m-1$, then there is a $\theta$ that satisfies both of the conditions given above.
By induction, we have that
\begin{equation*}\doublebar{\vec v}_{V^{p,\smooth}_m(V)}\leq C\int_V \abs{\nabla^m\vec v(x)}^p \dist(x,\partial V)^{p-1-p\smooth}\,dx.\end{equation*}
A standard patching argument shows that if $\Omega$ is a bounded Lipschitz domain and $\Tr_k^\Omega\vec u=0$ for all $0\leq k\leq m-1$, then 
\begin{equation*}\doublebar{\vec u}_{V^{p,\smooth}_m(\Omega)}\leq C\int_\Omega \abs{\nabla^m\vec u(x)}^p \dist(x,\partial \Omega)^{p-1-p\smooth}\,dx.\end{equation*}
Finally, by the bound~\eqref{eqn:unaveraged<averaged}, if $p\leq 2$ then
\begin{equation*}\int_\Omega \abs{\nabla^m\vec u(x)}^p \dist(x,\partial \Omega)^{p-1-p\smooth}\,dx
\leq C\doublebar{\vec u}_{\dot W^{p,\smooth}_{m,av}(\Omega)}^p\end{equation*}
as desired. This completes the proof of uniqueness.

We now establish existence of solutions. By \cite[Lemma~\ref*{P:lem:zero:boundary}]{Bar16pA}, we need only consider the case $\arr f=0$. 

Choose some $\arr H\in L^{p,\smooth}_{av}(\Omega)$. 
Again by the bound~\eqref{eqn:unaveraged<averaged}, and because $p\leq 2$, we have that
\begin{equation*}\int_\Omega \abs{\arr H(x)}^p \dist(x,\partial \Omega)^{p-1-p\smooth}\,dx
\leq C\doublebar{\arr H}_{L^{p,\smooth}_{av}(\Omega)}^p.\end{equation*}
Let $\Div_m\arr H$ be the distribution given by
\begin{equation*}
\bigl\langle \vec\varphi, \Div_m\arr H\bigr\rangle_\Omega
=
(-1)^m\bigl\langle\nabla^m\vec\varphi, \arr H\bigr\rangle_\Omega
.\end{equation*}
If $\vec\varphi \in V^{p',1-\smooth}_m(\Omega)$, then 
\begin{equation*}\langle \vec\varphi,\Div_m\arr H\rangle_\Omega 
= (-1)^m \int_\Omega \langle \nabla^m \vec \varphi(x),\arr H(x)\rangle \,dx\end{equation*}
and by H\"older's inequality and because $1-1/p-\smooth=-(1-1/p'-(1-\smooth))$,
\begin{align*}
\abs{\langle \vec\varphi,\arr H\rangle_\Omega}
&\leq \biggl(\int_\Omega \abs{\nabla^m\vec\varphi(x)}^{p'} \dist(x,\partial\Omega)^{1-1/p'-(1-\smooth)}\,dx\biggr)^{1/p'}
\\&\qquad\qquad\qquad\times\biggl(\int_\Omega \abs{\arr H(x)}^p \dist(x,\partial \Omega)^{p-1-p\smooth}\,dx\biggr)^{1/p}
\\
&\leq C\doublebar{\vec\varphi}_{V^{p',1-\smooth}_m(\Omega)} \doublebar{\arr H}_{L^{p,\smooth}_{av}(\Omega)}
.\end{align*}
Thus, $\Div_m\arr H\in V^{p,\smooth}_{-m}(\Omega)$. Thus, by Theorem~\ref{thm:MazMS10} there is some $\vec u\in V^{p,\smooth}_m(\Omega)$ with $L\vec u=\Div_m\arr H$ in~$V$ and with $\doublebar{\vec u}_{W^{p,\smooth}_m(\Omega)}\leq \doublebar{\arr H}_{L^{p,\smooth}_{av}(\Omega)}$.

It is clear that
\begin{equation*}\int_\Omega \abs{\nabla^m\vec u(x)}^p \dist(x,\partial \Omega)^{p-1-p\smooth}\,dx
\leq \doublebar{\vec u}_{W^{p,\smooth}_m(\Omega)}^p\leq C\doublebar{\arr H}_{L^{p,\smooth}_{av}(\Omega)}^p.\end{equation*}
By Lemma~\ref{lem:unaveraged}, we may improve to a $\dot W^{p,\smooth}_{m,av}(\Omega)$-norm on~$\vec u$, as desired.
\end{proof}

\bibliographystyle{amsalpha}
\bibliography{bibli}

\providecommand{\bysame}{\leavevmode\hbox to3em{\hrulefill}\thinspace}
\providecommand{\MR}{\relax\ifhmode\unskip\space\fi MR }
\providecommand{\MRhref}[2]{%
  \href{http://www.ams.org/mathscinet-getitem?mr=#1}{#2}
}
\providecommand{\href}[2]{#2}
\begin{thebibliography}{MMW11}

\bibitem[Agr07]{Agr07}
M.~S. Agranovich, \emph{On the theory of {D}irichlet and {N}eumann problems for
  linear strongly elliptic systems with {L}ipschitz domains}, Funktsional.
  Anal. i Prilozhen. \textbf{41} (2007), no.~4, 1--21, 96, English translation:
  Funct. Anal. Appl. \textbf{41} (2007), no.~4, 247--263. \MR{2411602
  (2009b:35070)}

\bibitem[AP98]{AdoP98}
Vilhelm Adolfsson and Jill Pipher, \emph{The inhomogeneous {D}irichlet problem
  for {$\Delta^2$} in {L}ipschitz domains}, J. Funct. Anal. \textbf{159}
  (1998), no.~1, 137--190. \MR{1654182 (99m:35048)}

\bibitem[AQ00]{AusQ00}
P.~Auscher and M.~Qafsaoui, \emph{Equivalence between regularity theorems and
  heat kernel estimates for higher order elliptic operators and systems under
  divergence form}, J. Funct. Anal. \textbf{177} (2000), no.~2, 310--364.
  \MR{1795955 (2001j:35057)}

\bibitem[Axe10]{Axe10}
Andreas Axelsson, \emph{Non-unique solutions to boundary value problems for
  non-symmetric divergence form equations}, Trans. Amer. Math. Soc.
  \textbf{362} (2010), no.~2, 661--672. \MR{2551501 (2010j:35103)}

\bibitem[Bar16a]{Bar16}
Ariel Barton, \emph{Gradient estimates and the fundamental solution for
  higher-order elliptic systems with rough coefficients}, Manuscripta Math.
  \textbf{151} (2016), no.~3-4, 375--418. \MR{3556825}

\bibitem[Bar16b]{Bar16pA}
\bysame, \emph{Perturbation of well posedness for higher order elliptic systems
  with rough coefficients}, \href{http://arxiv.org/abs/1604.00062}{\tt
  arXiv:1604.00062 [math.AP]}, March 2016.

\bibitem[Bar16c]{Bar16pB}
\bysame, \emph{Trace and extension theorems relating {B}esov spaces to weighted
  averaged {S}obolev spaces}, \href{http://arxiv.org/abs/1604.00058}{\tt
  arXiv:1604.00058 [math.FA]}, March 2016.

\bibitem[BHM]{BarHM15p}
Ariel Barton, Steve Hofmann, and Svitlana Mayboroda, \emph{Square function
  estimates on layer potentials for higher-order elliptic equations}, Math.
  Nachr., to appear (preprint available at
  \href{http://arxiv.org/abs/1508.04988v1}{\tt arXiv:1508.04988v1 [math.AP]}).

\bibitem[BHM17]{BarHM17pC}
\bysame, \emph{The {N}eumann problem for higher order elliptic equations with
  symmetric coefficients}, \href{http://arxiv.org/abs/1703.06962}{\tt
  arXiv:1703.06962 [math.AP]}, March 2017.

\bibitem[BL76]{BerL76}
J{\"o}ran Bergh and J{\"o}rgen L{\"o}fstr{\"o}m, \emph{Interpolation spaces.
  {A}n introduction}, Springer-Verlag, Berlin, 1976, Grundlehren der
  Mathematischen Wissenschaften, No. 223. \MR{0482275 (58 \#2349)}

\bibitem[BM16a]{BarM16B}
Ariel Barton and Svitlana Mayboroda, \emph{Higher-order elliptic equations in
  non-smooth domains: a partial survey}, Harmonic Analysis, Partial
  Differential Equations, Complex Analysis, Banach Spaces, and Operator Theory
  (Volume 1). Celebrating Cora Sadosky's life, Association for Women in
  Mathematics Series, vol.~4, Springer-Verlag, 2016, pp.~55--121.

\bibitem[BM16b]{BarM16A}
\bysame, \emph{Layer potentials and boundary-value problems for second order
  elliptic operators with data in {B}esov spaces}, Mem. Amer. Math. Soc.
  \textbf{243} (2016), no.~1149, v+110. \MR{3517153}

\bibitem[Cam80]{Cam80}
S.~Campanato, \emph{Sistemi ellittici in forma divergenza. {R}egolarit\`a
  all'interno}, Qua\-der\-ni. [Publications], Scuola Normale Superiore Pisa,
  Pisa, 1980. \MR{668196 (83i:35067)}

\bibitem[CG85]{CohG85}
Jonathan Cohen and John Gosselin, \emph{Adjoint boundary value problems for the
  biharmonic equation on {$C^1$} domains in the plane}, Ark. Mat. \textbf{23}
  (1985), no.~2, 217--240. \MR{827344 (88d:31006)}

\bibitem[CMS85]{CoiMS85}
R.~R. Coifman, Y.~Meyer, and E.~M. Stein, \emph{Some new function spaces and
  their applications to harmonic analysis}, J. Funct. Anal. \textbf{62} (1985),
  no.~2, 304--335. \MR{791851 (86i:46029)}

\bibitem[DK90]{DahK90}
B.~E.~J. Dahlberg and C.~E. Kenig, \emph{{$L^p$} estimates for the
  three-dimensional systems of elastostatics on {L}ipschitz domains}, Analysis
  and partial differential equations, Lecture Notes in Pure and Appl. Math.,
  vol. 122, Dekker, New York, 1990, pp.~621--634. \MR{1044810 (91h:35053)}

\bibitem[FMM98]{FabMM98}
Eugene Fabes, Osvaldo Mendez, and Marius Mitrea, \emph{Boundary layers on
  {S}obolev-{B}esov spaces and {P}oisson's equation for the {L}aplacian in
  {L}ipschitz domains}, J. Funct. Anal. \textbf{159} (1998), no.~2, 323--368.
  \MR{1658089 (99j:35036)}

\bibitem[JK95]{JerK95}
David Jerison and Carlos~E. Kenig, \emph{The inhomogeneous {D}irichlet problem
  in {L}ipschitz domains}, J. Funct. Anal. \textbf{130} (1995), no.~1,
  161--219. \MR{1331981 (96b:35042)}

\bibitem[MM04]{MayMit04A}
Svitlana Mayboroda and Marius Mitrea, \emph{Sharp estimates for {G}reen
  potentials on non-smooth domains}, Math. Res. Lett. \textbf{11} (2004),
  no.~4, 481--492. \MR{2092902 (2005i:35059)}

\bibitem[MM13a]{MitM13B}
Irina Mitrea and Marius Mitrea, \emph{Boundary value problems and integral
  operators for the bi-{L}aplacian in non-smooth domains}, Atti Accad. Naz.
  Lincei Rend. Lincei Mat. Appl. \textbf{24} (2013), no.~3, 329--383.
  \MR{3097019}

\bibitem[MM13b]{MitM13A}
\bysame, \emph{Multi-layer potentials and boundary problems for higher-order
  elliptic systems in {L}ipschitz domains}, Lecture Notes in Mathematics, vol.
  2063, Springer, Heidelberg, 2013. \MR{3013645}

\bibitem[MMS10]{MazMS10}
V.~Maz'ya, M.~Mitrea, and T.~Shaposhnikova, \emph{The {D}irichlet problem in
  {L}ipschitz domains for higher order elliptic systems with rough
  coefficients}, J. Anal. Math. \textbf{110} (2010), 167--239. \MR{2753293
  (2011m:35088)}

\bibitem[MMW11]{MitMW11}
I.~Mitrea, M.~Mitrea, and M.~Wright, \emph{Optimal estimates for the
  inhomogeneous problem for the bi-{L}aplacian in three-dimensional {L}ipschitz
  domains}, J. Math. Sci. (N. Y.) \textbf{172} (2011), no.~1, 24--134, Problems
  in mathematical analysis. No. 51. \MR{2839870 (2012h:35056)}

\bibitem[PV92]{PipV92}
Jill Pipher and Gregory Verchota, \emph{The {D}irichlet problem in {$L^p$} for
  the biharmonic equation on {L}ipschitz domains}, Amer. J. Math. \textbf{114}
  (1992), no.~5, 923--972. \MR{1183527 (94g:35069)}

\bibitem[Tol16]{Tol16p}
Patrick Tolksdorf, \emph{$\mathcal{R}$-sectoriality of higher-order elliptic
  systems on general bounded domains},
  \href{http://arxiv.org/abs/1611.05663v1}{\tt arXiv:1611.05663v1 [math.AP]},
  November 2016.

\bibitem[Ver05]{Ver05}
Gregory~C. Verchota, \emph{The biharmonic {N}eumann problem in {L}ipschitz
  domains}, Acta Math. \textbf{194} (2005), no.~2, 217--279. \MR{2231342
  (2007d:35058)}

\bibitem[Ver10]{Ver10}
\bysame, \emph{Boundary coerciveness and the {N}eumann problem for 4th order
  linear partial differential operators}, Around the research of {V}ladimir
  {M}az'ya. {II}, Int. Math. Ser. (N. Y.), vol.~12, Springer, New York, 2010,
  pp.~365--378. \MR{2676183 (2011h:35065)}

\bibitem[Zan00]{Zan00}
Daniel~Z. Zanger, \emph{The inhomogeneous {N}eumann problem in {L}ipschitz
  domains}, Comm. Partial Differential Equations \textbf{25} (2000), no.~9-10,
  1771--1808. \MR{1778780 (2001g:35056)}

\end{thebibliography}
\end{document}